\documentclass[12pt,a4paper,reqno]{amsart}

\usepackage[latin1]{inputenc}
\usepackage{amssymb,amsmath,amstext,amsthm,amsfonts}
\usepackage{graphicx}
\usepackage{multicol}
 
\usepackage{hyperref}

\makeatletter
\newcommand{\mylabel}[2]{#2\def\@currentlabel{#2}\label{#1}}
\makeatother
\usepackage{a4wide}
\usepackage{booktabs} 
\usepackage{multirow}

\usepackage{comment}
\usepackage[]{inputenc}
\usepackage{amssymb,amsmath,amstext,amsthm,amsfonts, amscd}
\usepackage{mathtools}
\usepackage{graphicx}
\usepackage{multicol}
\usepackage{hyperref}

\newcommand{\ols}[1]{\mskip.5\thinmuskip\overline{\mskip-.5\thinmuskip {#1} \mskip-.5\thinmuskip}\mskip.5\thinmuskip} 
\newcommand{\olsi}[1]{\,\overline{\!{#1}}} 
\makeatletter
\newcommand\closure[1]{
  \tctestifnum{\count@stringtoks{#1}>1} 
  {\ols{#1}} 
  {\olsi{#1}} 
}
\long\def\count@stringtoks#1{\tc@earg\count@toks{\string#1}}
\long\def\count@toks#1{\the\numexpr-1\count@@toks#1.\tc@endcnt}
\long\def\count@@toks#1#2\tc@endcnt{+1\tc@ifempty{#2}{\relax}{\count@@toks#2\tc@endcnt}}
\def\tc@ifempty#1{\tc@testxifx{\expandafter\relax\detokenize{#1}\relax}}
\long\def\tc@earg#1#2{\expandafter#1\expandafter{#2}}
\long\def\tctestifnum#1{\tctestifcon{\ifnum#1\relax}}
\long\def\tctestifcon#1{#1\expandafter\tc@exfirst\else\expandafter\tc@exsecond\fi}
\long\def\tc@testxifx{\tc@earg\tctestifx}
\long\def\tctestifx#1{\tctestifcon{\ifx#1}}
\long\def\tc@exfirst#1#2{#1}
\long\def\tc@exsecond#1#2{#2}
\makeatother

\newcommand{\qand}{\quad\text{and}\quad}

\def\inte{\operatorname{int}}

\def\dist{\operatorname{dist}}

\def\leb{m}
\def\emb{\operatorname{Emb}}

\newtheorem{maintheorem}{Theorem}

\newcommand{\cmt}{\begin{maintheorem}}
\newcommand{\fmt}{\end{maintheorem}}

\newtheorem{maincorollary}[maintheorem]{Corollary}

\newcommand{\cmc}{\begin{maincorollary}}
\newcommand{\fmc}{\end{maincorollary}}

\newtheorem{lemma}{Lemma}[section]
\newtheorem{theorem}[lemma]{Theorem}
\newtheorem{corollary}[lemma]{Corollary}
\newtheorem{proposition}[lemma]{Proposition}

\theoremstyle{remark}
\newtheorem{remark}[lemma]{Remark}

\DeclareMathOperator{\Leb}{Leb}

\DeclarePairedDelimiter\abs{\lvert}{\rvert}%

\makeatletter
\let\oldabs\abs
\def\abs{\@ifstar{\oldabs}{\oldabs*}}

\makeatletter
\@namedef{subjclassname@2020}{\textup{2020} Mathematics Subject Classification}
\makeatother

\thanks{The authors were partially supported by     CMUP (UID/MAT/00144/2019),  PD/BD/128062/2016 and PTDC/MAT-PUR/28177/2017, which are funded by FCT (Portugal) with national (MEC) and European structural funds through the program  FEDER, under the partnership agreement PT2020. 
}

\begin{document}
\title[Entropy formula for systems with inducing schemes]{Entropy formula for  systems\\ with inducing schemes}
\date{}

\author[J. F. Alves]{Jos\'{e} F. Alves}
\address{Jos\'{e} F. Alves\\ Centro de Matem\'{a}tica da Universidade do Porto\\ Rua do Campo Alegre 687\\ 4169-007 Porto\\ Portugal}
\email{jfalves@fc.up.pt} \urladdr{http://www.fc.up.pt/cmup/jfalves}

\author[D. Mesquita]{David Mesquita}
\address{David Mesquita\\ Centro de Matem\'{a}tica da Universidade do Porto\\ Rua do Campo Alegre 687\\ 4169-007 Porto\\ Portugal}

\maketitle




\begin{abstract}
We obtain entropy formulas for  SRB measures with finite entropy given by inducing schemes. In the first part of the work, we obtain Pesin entropy formula for the class of noninvertible systems whose SRB measures are given by   Gibbs-Markov induced maps. In the second part, we obtain Pesin entropy formula   for invertible maps whose SRB measures given by Young sets,  taking into account a classical compression technique along the stable direction that allows a reduction of the return map associated with a Young set to a Gibbs-Markov map. In both cases, we give applications of  our main results  to several  classes of dynamical systems with singular sets, where the classical results by Ruelle and Pesin  cannot be applied. We also present examples  of systems with SRB measures given by inducing schemes  for which Ruelle inequality  does not hold. 
 
\end{abstract}
\maketitle

\addtocontents{toc}{\protect\setcounter{tocdepth}{0}}
\section*{Introduction}

\addtocontents{toc}{\protect\setcounter{tocdepth}{2}}

The concept of \emph{entropy} was   introduced in  Dynamical Systems by Kolmogorov   in \cite{K58}, based on an analogous notion in  Information Theory proposed by Shannon in~\cite{S48}. In broad terms, entropy  measures the exponential rate at which the dynamical complexity increases as the system is iterated in time, and thus relates to the unpredictability of the system. A very natural first issue arising from the definition is the estimation of entropy  for concrete systems -- a question already treated by Kolmogorov and Sinai themselves in their celebrated theorem concerning generating partitions --,  finding one of its first general answers in the \emph{Rohlin formula}, which expresses the entropy in terms of the integral of the Jacobian with respect to an invariant measure, and in   subsequent works \cite{AR62,BC92,LW77}.

 For smooth diffeomorphisms of a Riemannian manifold, Ruelle established in \cite{R78} that the entropy of an  invariant probability measure is always bounded by the integral of the sum of the positive Lyapunov exponents (counted with multiplicity) with respect to that measure. Margulis first established  this inequality for diffeomorphisms preserving a smooth measure.  The reverse inequality was obtained  by Pesin in \cite{P77},  for  \emph{Sinai-Ruelle-Bowen (SRB) measures}, i.e. invariant probability measures whose conditional measures  are  absolutely continuous with respect to the Lebesgue measure along local unstable manifolds. A simpler proof  of Pesin inequality was given  by  Ma\~n\'e  in \cite{M81a}. A characterisation of   the entropy formula in terms of the SRB property was  given by Ledrappier and Young in \cite{LY85};  see also~\cite{Y03}. Natural versions for noninvertible smooth maps (endomorphisms) have been drawn by Liu, Qian et al  in \cite{L98,LQ95,QXZ09}. 
 
 There is currently a vast literature addressing Ruelle inequality and Pesin entropy formula for smooth dynamical systems.  Extensions of these results were obtained by Ledrappier and Strelcyn in~\cite{LS82} for the class of maps having points with infinite derivative introduced by Katok and Strelcyn  in~\cite{KSL86}, inspired on billiard maps. Albeit, to the best of our knowledge, besides  \cite{LS82} and the recent work \cite{AP21}, not much is known on the existence of entropy formulas that can be applied   to piecewise smooth maps with singular sets in general, especially in dimension greater than one. For one-dimensional dynamical systems, see e.g. \cite{BJ12,DKU90, H91, K89,L81,L20,R64}. 
The utility of the entropy formula can be detected, for example, in the works~\cite{ACF10,AOT06,AP21,L19}, where it is used in an important way to prove the continuity of entropy in certain families of transformations.

Young has shown  that the statistical properties -- in particular, the existence of SRB measures -- of some nonuniformly hyperbolic dynamical systems can be deduced by means of inducing schemes. This means that we chose a region of the phase space and define a new dynamical system in that region with some good analytical/geometric properties, considering an appropriate return transformation (not necessarily the first one) in subdomains of the considered region. An abstract framework for this strategy was developed quite successfully in \cite{Y98} for systems with contractive directions and in \cite{Y99} for systems without contractive directions.  In recent years, the  results  by  Young  have been applied by many authors to various classes of dynamical systems, comprising, in particular, systems with singularities (including billiards). A natural question  is:
\begin{quote}\em To what extent the existence of an inducing scheme, by itself, guarantees  an entropy formula for the SRB measure given by that inducing scheme?
\end{quote}
 In our main results, we show that the answer is affirmative  whenever the entropy of the SRB measure is finite. We  also characterise the finiteness  of the entropy   in terms of analytical properties of   the inducing scheme, and give examples of SRB measures given by inducing schemes  (necessarily with infinite entropy)  for which Ruelle inequality does not hold. 

Comparing our results with those   by Katok, Ledrappier and Strelcyn in the aforementioned works \cite{KSL86,LS82}, the  advantage of our approach  consists essentially of two main points:  \emph{i)} we do not have  any assumptions    on the SRB measure, only analytical and geometric properties of the induced scheme, unlike  (1.1) and (1.2) in~\cite{KSL86} or   conditions   2.1 and  2.2 in~\cite{LS82};  
 \emph{ii)}~in recent years,   inducing schemes have become a widely used tool, the existence of which has been demonstrated  for many classes  of dynamical systems. A vast list of examples of maps with singular sets are given at Sections~\ref{se.app1} and~\ref{se.app2}, illustrating applications of our main results. In terms of applications of the results in \cite{KSL86, LS82} (only to billiards), the fact that the density of the SRB measure of a billiard map has a simple and well-known  expression  is used in an important way; see 
  \cite[Part V]{KSL86}.
\addtocontents{toc}{\protect\setcounter{tocdepth}{0}}

\subsubsection*{Overview.}
 \addtocontents{toc}{\protect\setcounter{tocdepth}{2}}
This paper is organised as follows. In Part~\ref{part1}, we obtain the entropy formula  for maps without contracting directions whose SRB measures are given by   Gibbs-Markov induced  maps. These concepts, as well as the main results of this part, are presented    in Section~\ref{Markovstruc}. Proofs of the   results are provided  in the subsequent sections. In Section~\ref{se.systemno}, we give an example of a  piecewise smooth map with infinite entropy   for which the integral of the Jacobian is finite. 
This example 
  illustrates the failure of two classical results under  slightly  more general assumptions  than those usually required, even for the somewhat regular SRB measures given by   Gibbs-Markov induced maps:
\begin{itemize}
\item[(1)] 
\em Ruelle inequality does not hold  if   the differentiability of the transformation is assumed only  almost everywhere;
\item[(2)]  
  the conclusion of Shannon-McMillan-Breiman Theorem is no longer valid
 if    the generating partition is not assumed with finite entropy.
\end{itemize}
This is explained in detail in Remarks~\ref{re.smbt2} and~\ref{re.smbt}.
In Section~\ref{se.app1},
we apply the main results of the first part  of this work to some classes of piecewise smooth maps with SRB measures given by Gibbs-Markov induced  maps.
 In Part~\ref{part2}, we obtain the entropy formula  for piecewise smooth diffeomorphisms  with  contracting directions whose SRB measures are given by  Young sets. These concepts, as well as the main results of this part, are provided   in Section~\ref{section.GMY}. The proofs of the   results are given in the subsequent sections. In Section~\ref{se.systemno2} we give an example of  system with infinite entropy   for which Ruelle inequality does not hold. In Section~\ref{se.app2},
we apply our main results of the second part  to some classes of piecewise smooth diffeomorphisms with SRB measures given by Young sets.

  \tableofcontents

  \addtocontents{toc}{\protect\setcounter{tocdepth}{0}}

 \subsubsection*{Acknowledgements} The authors  are grateful to Jérôme Buzzi, Maria Carvalho, Mark Demers, Stefano Luzzatto  and Marcelo Viana for    valuable discussions on these topics and for providing   some useful references.
 
 \addtocontents{toc}{\protect\setcounter{tocdepth}{2}}


\part{Systems with expanding structures}\label{part1}
Let   $M$ be a  Riemannian manifold and let $m$ be the Lebesgue measure  on the Borel sets of~$M$. 
Assume that $f:M\to M$ is a   piecewise $C^{1+\eta}$ map, meaning that there are  at most countably many  pairwise disjoint open regions $M_1,  M_2,\dots$ 
such that
  $\bigcup_{i=1}^\infty\overline{M}_i=M$ and   $f\vert_{\cup_{i=1}^\infty M_i}$ is a $C^{1+\eta}$ map.
  We will refer to
   $$S= M\setminus \cup_{i=1}^\infty M_i$$
   as the \emph{singular  set} of $f$ and assume that $m(S)=0$. 
Typically, $S$ is a set of  discontinuity  points or a set of points where the derivative of $f$ does not exist (possibly being infinite).

\section{Gibbs-Markov induced maps} \label{Markovstruc}

Assume that $\Delta_0\subset M\setminus S$ is a Borel set with $m(\Delta_0)>0$. For simplicity, the restriction of $m$ to $\Delta_0$ will still be denoted by $m$. 
Consider
  a countable $m$ mod $0$ partition $\mathcal{P}$ of $\Delta_0$ into open sets of $\Delta_0$ and 
  a function   $R:\Delta_0 \rightarrow \mathbb{N}$   constant on each  element of $\mathcal{P}$ such that, for all $\omega \in \mathcal{P}$,
\begin{itemize} 
\item $f^j(\omega)\cap S=\emptyset$, for all  $1\le j\le R(\omega)$;
\item $f^{R(\omega)}(\omega) \subset  \Delta_0 $. \label{induced map}
\end{itemize}
 We associate to these objects a    map $F:\Delta_0\to\Delta_0$, setting 
$$
F|_{\omega}=f^{R(\omega)}|_\omega,\quad\text{for each $\omega \in \mathcal{P}$}.
$$
 The map  $F$ will be frequently  denoted  by $f^R$ and called an \textit{induced map} for $f$; the function~$R$ will be called the \textit{recurrence time} associated with the induced map.   
   We say that $F$ is  a
 \emph{Gibbs-Markov}\index{Gibbs-Markov map!weak} map  
 if  
   conditions    \ref{G1}-\ref{G5} below  are satisfied.

   \begin{itemize}
       \item[\mylabel{G1}{(G$_1$)}]        \emph{Markov}: \index{Markov} $F$ maps each $ \omega\in  \mathcal P  $ bijectively to    $\Delta_0$ ($m$ mod 0).
     \item[\mylabel{G2}{(G$_2$)}]       \emph{Nonsingular}: \index{nonsingular}     there  exists  a measurable  function  $J_F:\Delta_0\to (0,\infty)$   such that, for every measurable set  $A\subset \omega\in\mathcal P$,
 $$m(F(A))=\int_A J_F dm.$$
 \end{itemize}
 The function $J_F$ is  called the  \emph{Jacobian}\index{Jacobian}  of $F$  (with respect to $m$). The next two properties are related to the dynamical partitions generated by $\mathcal P$.
 Set for each  $n\ge1$
  \begin{equation}\label{eq.novacoisa}
\mathcal P_n=\bigvee_{i=0}^{n-1}F^{-i}\mathcal P=
 \left\{\omega_0\cap F^{-1}(\omega_1)\cap\cdots\cap F^{-(n-1)}(\omega_{n-1})\, \colon  \omega_0,\dots,\omega_{n-1}\in\mathcal P\right\}
\end{equation}
and 
  \begin{equation*}\label{eq.novacoisa2}
\mathcal P_\infty=\bigvee_{i=0}^{\infty}F^{-i}\mathcal P=
 \left\{\omega_0\cap F^{-1}(\omega_1)\cap F^{-2}(\omega_2)\cap\cdots \colon \omega_n\in\mathcal P\text{ for all $n\ge0$}\right\}.
\end{equation*}
 \begin{itemize}
        \item[\mylabel{G3}{(G$_3$)}]  \emph{Generating}: \index{separability} 
        the $\sigma$-algebra generated by    $\bigcup_{n=1}^{\infty} \mathcal P_n$ coincides with $\mathcal A$ ($m$ mod 0).  
 \end{itemize}
  \begin{itemize}
        \item[\mylabel{G4}{(G$_4$)}]  \emph{Separating}: \index{separability} 
       $ \mathcal P_\infty $  is the partition into single points of $\Delta_0$  ($m$ mod 0).  
 \end{itemize}
 It follows from~\ref{G4} that
   the \emph{separation time}\index{separation!time}
       \begin{equation}\label{eq.separa}
 s(x,y)=\min\big\{  n\ge 0 \colon \text{$F^n(x)$ and $F^n(y)$ lie in distinct elements of $\mathcal P$}\big\}
\end{equation}
is well defined and finite for  distinct points 
  $x,y$     in a  full $m$ measure subset  of $\Delta_0$. 
For definiteness, set the separation time equal to zero for all  other points. 
      \begin{itemize}
  \item[\mylabel{G5}{(G$_5$)}]      \emph{Gibbs}: \index{Gibbs} 
   there are $C>0$  and $0<\beta<1$  such that,   for all $x,y\in \omega\in\mathcal P$,    
    $$\log\frac{J_F(x)}{J_F(y)}\le  C\beta^{s(F(x),F(y))}.$$
\end{itemize}

We say that the induced map $f^R$ has \emph{integrable  recurrence times} if   $R \in L^1(m)$.
The next result is standard for transformations admitting  Gibbs-Markov induced maps; see e.g.~\cite[Theorem 3.13]{A20} and~\cite[Corollary 3.21]{A20}. In the present  context, by an \emph{SRB measure}\index{measure!SRB}  we mean  an invariant probability measure that is absolutely continuous with respect to the Lebesgue measure~$m$.

\begin{theorem}\label{th.SRB0} 
If $f^R:\Delta_0\to \Delta_0$ is a Gibbs-Markov induced map  for $f$ with integrable recurrence times, then
\begin{enumerate}
\item $f^R$ has a  unique ergodic  SRB measure $\nu_0$;
\item $f$ has a unique ergodic SRB measure $\mu$ with $\mu(\Delta_0)>0$,  which is given by  
 $$ 
{\mu}=\frac1{\sum_{j= 0}^\infty\nu_0(\{  R > j\})} \sum_{j=0}^\infty f^j_{*}(\nu_0|\{  R > j\}).
$$ 
\end{enumerate}
\end{theorem}

In these circumstances, we say that the SRB measure $\mu$ is  \emph{given by the Gibbs-Markov induced map~$f^R$}. 
Regarding the integral  the entropy formula, in the present context we are naturally lead to consider the case where all Lyapunov exponents are positive and their sum   coincides with the Jacobian $|\det Df|$ of the map $f$ with respect to the reference (Lebesgue) measure~$m$. As shown in \cite[Proposition 2.5]{LS82}, this happens under very general conditions. In our first main result, we establish an entropy formula for an SRB measure~$\mu$  with finite entropy~$h_\mu(f)$ given by a Gibbs-Markov induced map; 
we also characterise  those SRB measures  with finite  entropy.


%


\begin{maintheorem} \label{th.formendo}
Let $f:M\to M$ be a   piecewise $C^{1+\eta}$ map with 
   an ergodic SRB measure $\mu$ given by a Gibbs-Markov induced map $f^R$. Then,
 \begin{enumerate}
\item  if $h_\mu(f)<\infty$, then
$$
h_\mu(f)=\int_M   \log |\det Df| \,  d\mu;
$$
\item $h_\mu(f)<\infty$ if, and only if,
  $$\int_{\Delta_0} R\,\log |\det D{f^R}| \,dm<\infty . 
$$
\end{enumerate}

\end{maintheorem}


%


  The strategy for proving the first item in Theorem~\ref{th.formendo} is based on building  a tower extension $(T,\nu)$  for $(f,\mu)$ and   showing that
$$
h_\mu(f) =h_\nu(T) =\int \log J_T d\nu = \int \log |\det Df| d\mu.
$$
The first equality   is a consequence of a general result due to Buzzi  for extensions with countably many fibers. The second and third      equalities will be  deduced   in~Proposition~\ref{entrinduce}
 and Lemma~\ref{vvvvv}, respectively. The second item of Theorem~\ref{th.formendo} will be   obtained    in Lemma~\ref{le.giniteq}.

Assuming that there is  some constant $C>0$ for which $|\det Df |\le C$, 
 it follows from the chain rule   that
  $ |\det D{f^R}|\le C^R.$ We therefore have the following simple, albeit  useful, consequence of Theorem~\ref{th.formendo}.

\begin{maincorollary} \label{co.entform0}
Let $f:M\to M$ be a   piecewise $C^{1+\eta}$ map 
with an ergodic SRB measure~$\mu$ given by a Gibbs-Markov induced map $f^R$. If $| \det Df|$ is bounded   and $\int_{\Delta_0}R^2 dm<\infty$, then
$$
h_\mu(f)=\int_M   \log |\det Df| \,  d\mu<\infty.
$$
\end{maincorollary}
 
In Section~\ref{se.systemno}, we provide   an example of a  piecewise~$C^\infty$ interval map~$f$  with an SRB  measure~$\mu$ given by an induced map for which
the   formula in the  first item of Theorem~\ref{th.formendo} is no longer valid in the infinite entropy case. 
 In Section~\ref{se.app1}, we apply   Theorem~\ref{th.formendo} and Corollary~\ref{co.entform0} to some classes of piecewise $C^{1+\eta}$ maps with nonempty singular sets.

\section{Tower system} \label{towerextension}
In this section, we follow ideas in \cite{Y98,Y99}  and introduce the tower system associated with   an induced map and recall some useful properties of this new dynamical system.  The construction of the tower   will be done  with more generality than needed in this section, for we be able to apply it in Section~\ref{sub.quotow}. 
Consider
\begin{itemize}
\item a  Gibbs-Markov map $F:\Delta_0\to\Delta_0$  on  a  space $\Delta_0$ with a  reference measure~$m$;
\item  the  countable     mod 0 partition  $\mathcal P$    of $\Delta_0$ associated with $F$;
\item  a   measurable function $R:\Delta_0 \to \mathbb N$   constant on the elements of $\mathcal P$. 
\end{itemize}
%
%
%
We associate to these objects the \textit{tower} 
$$
\Delta=\{(x,\ell) \, : \, x\in \Delta_0 \mbox{ and }  0\leq \ell < R(x) \}
$$
and the \textit{tower map} $T:\Delta \rightarrow \Delta$, given by

\[
T(x,\ell)=
                \begin{cases} 
                  (x,\ell+1),   & \mbox{if $\ell<R(x)-1$};\\
                  (F(x),0), &  \mbox{if $\ell=R(x)-1$}.
                \end{cases}
\]

\noindent
For $\ell\in\mathbb N_0$, the  \textit{$\ell^{th}$-level of the tower} is  
\begin{equation} \label{llevel}
\Delta_\ell=\{(x,\ell) \in \Delta \}.
\end{equation}
 Observe that we use  $\Delta_0$ to represent both the $0^{th}$-level  (or \textit{base}) of the tower and the   domain of $F$ upon which the tower is built, since they are naturally identified with each other.  Moreover, under this identification it is straightforward to check  that $T^{R(x)}(x,0)=F(x)$, for each $x  \in \Delta_0$. We call $T^R:\Delta_0 \rightarrow \Delta_0$ the \textit{return to the base}, which in this case   is actually a  first return map. 
 In the same spirit, the $\ell^{th}$ level of the tower $\Delta_\ell$  is also naturally identified with the set $ \{ R >\ell\} \subseteq \Delta_0$. This identification  allows us to extend the reference measure $m$ on $\Delta_0$ to a measure on $\Delta$, that we still denote by $m$. It  happens that 
 $$
m(\Delta)=\sum_{\ell\geq 0} m(\Delta_\ell)=\sum_{\ell\geq 0} m(\{R>\ell\}) 
=\int_{\Delta_0}\, R \, dm,
$$ 
and so, the integrability of $R$ with respect to $m$ (on $\Delta_0$) is a necessary and sufficient condition for the finiteness of the measure $m$ on $\Delta$. The next result provides a unique $T$-invariant probability measure which is absolutely continuous with respect to $m$; 
see e.g.~\cite[Theorem 3.24]{A20} for a proof.

\begin{theorem}\label{th.measuretower} 
 Let $T:\Delta\to\Delta$ be the tower map associated with a  Gibbs-Markov map $F$ and recurrence times $R\in L^1(m)$. 
If $\nu_0$ is the unique $F$-invariant probability measure  with $\nu_0\ll m$, then
$$ 
 {\nu}=\frac1{ \sum_{j=0}^{\infty}  \nu_0\{  R > j\}} \sum_{j=0}^{\infty} T^j_{*}( \nu_0|\{  R > j\})
$$
is the unique  $T$-invariant probability measure such that $\nu\ll m$.   
Moreover,  $\nu$ is ergodic and the density $d\nu/dm$ is bounded from above and below by positive constants.
\end{theorem}
Set for convenience
\begin{equation}\label{eq.rho}
\rho=\sum_{j=0}^{\infty}  \nu_0\{  R > j\}=\int R d\nu_0<\infty.
\end{equation}
The finiteness of this quantity is due to the fact we assume $R\in L^1(m)$ and $d\nu_0/dm$ is bounded from above and below by positive constants.
An important feature of the tower construction    is that 
 \begin{equation}\label{eq.pi1}
 \begin{array}{rccl}
 \pi \colon &\!\!\!\!   \Delta & \longrightarrow &
 M
 \\
& \!\!(x,\ell)&\longmapsto & f^\ell(x),
 \end{array}
 \end{equation}
   is a semiconjugacy between the original system and the tower system, i.e. 
\begin{equation*}\label{eq.cnjugar}
f\circ \pi = \pi\circ  T\qand \mu=\pi_*\nu,
\end{equation*}
%
where $\mu$ is the  $f$-invariant  probability  measure  given by 
  Theorem~\ref{th.SRB0}; see e.g. \cite[Proposition 3.27]{A20}. This means that   the tower system $(T,\nu)$ is an \emph{extension} of the   system $(f,\mu)$, or  $(f,\mu)$ is a \emph{factor} of   $(T,\nu)$.

\section{Jacobians and natural partitions}\label{jacandnat}

It is straightforward to check that   the tower map $T:\Delta\to\Delta$ also has a Jacobian $J_T$ (with respect to the measure $m$ on $\Delta_0$), given by
 \begin{equation}\label{eq.jacob}
J_T(x,\ell)= 
                \begin{cases}
                  1,  &  \mbox{if $\ell<R(x)-1$};\\
                  J_{F}(x),  & \mbox{if $\ell=R(x)-1$}.
                \end{cases}
            \end{equation}                
We will deduce some properties for the Jacobian $J_T$  and relate it to  a natural partition associated with  the tower map in a Volume Lemma. 
%
%
We start with some simple results for the Gibbs-Markov map~$F:\Delta_0\to\Delta_0$. 

\begin{lemma}\label{le.lower} There exists $C>0$ such that, for all $\omega\in\mathcal P$ and   $x\in \omega$,
 $$\frac1C\le \nu_0(\omega)J_F(x)\le C.$$
\end{lemma}

\begin{proof}
%
Since  $F$ is a Gibbs-Markov map, it follows from \ref{G5} that there exists  some   $C_1>0$ such that, for all $\omega\in\mathcal P$ and  $x,y\in\omega$,
\begin{equation}\label{eq.quot}
\frac1{C_1}\le \frac{J_{F}(y)}{J_{F}(x)}\le C_1.
\end{equation}
Using \ref{G2}, we obtain   for all $\omega\in\mathcal P$ and  $x\in\omega$,
$$m(\Delta_0)=\int_\omega J_{F}(y) dm(y)=\int_\omega J_{F}(x)\frac{J_{F}(y)}{J_{F}(x)} dm(y),
 $$
which together with   \eqref{eq.quot} yields
 \begin{equation*}\label{enqua}
\frac1{C_1} m(\omega) J_F(x)\le m(\Delta_0)  \le C_1 m(\omega) J_F(x).
\end{equation*}
Considering that  $d\nu_0/dm$ is bounded from above and below by positive constants, we easily get  the conclusion.
\end{proof}

\begin{corollary}\label{co.lower} There exists $c>0$ such that $J_F(x)\ge c$, for $\nu_0$ almost all $x\in \Delta_0$.
\end{corollary}
\begin{proof}
It follows from Lemma~\ref{le.lower}  that, for all $\omega\in\mathcal P$ and   $x\in \omega$,
 $$J_F(x)\ge \frac{1}{C\nu_0(\omega)}\ge \frac1C.$$
Recall  that   $\mathcal P$ is an $m$ mod 0 partition of the set~$\Delta_0$ and $\nu_0$ is  a probability measure  equivalent to~$m$. 
\end{proof}

Let $(X,\eta)$ be a measure space and  $\varphi:X\to\mathbb R$ a measurable function. Recall that the integral  of $\varphi$ with respect to $\eta$    is defined   whenever 
 \begin{equation*}\label{star}
\int_{X}\varphi^+ d\eta<\infty\quad\text{or}\quad \int_{X} \varphi^- d\eta<\infty,
\end{equation*}
where $\varphi^\pm =\max\{0,\pm \varphi \}$.  In such case,
 \begin{equation}\label{star2}
\int_X\varphi d\eta=\int_{X}\varphi^+ d\eta- \int_{X} \varphi^- d\eta .
\end{equation}
As a consequence of the  next result, we have  that  the integral  of $  \log J_{F}$ with respect to the probability measure $\nu_0$ is defined.

\begin{corollary}\label{co.integral0} 
$\displaystyle\int_{\Delta_0} \log  {J_{F}}^{-} d\nu_0<\infty$.
\end{corollary}

\begin{proof}
It follows from Corollary~\ref{co.lower} that  $ -  \log J_{\!F}(x) \le -\log \alpha$, for $\nu_0$ almost all $x\in\Delta_0$. This implies that    $   \log {J_{F}}^-(x)\le \max\{0,-\log \alpha\}$, for~$\nu_0$ almost every $x\in \Delta_0$. Since $\nu_0$ is a probability measure, we have   the conclusion. 
\end{proof}

%

  In the next result we relate the Jacobians of the maps $F$ and $T$ with respect to the reference measure. Recall that $\rho>0$ has been introduced in~\eqref{eq.rho}. 

\begin{lemma} \label{davids}
$\displaystyle\int_\Delta  \log  {J_T}^\pm \, d{\nu} =\frac1\rho  \int_{\Delta_0}   \log {J_F}^\pm \,  d \nu_0.$
\end{lemma}
\begin{proof}
For each  $ \ell\in \mathbb{N}_0$, set 
$$
\Delta_\ell^{\ell+1}=\{(x,\ell) \in \Delta \colon  R(x)=\ell+1 \}.
$$ 
Using these sets and~\eqref{eq.jacob}, we obtain
$$
 \log{J_T}^\pm(x,\ell)= 
                \begin{cases}
                  0,  &  \mbox{if $x\notin \Delta_\ell^{\ell+1}$};\\
                  \log{J_F}^\pm(x),  & \mbox{if $x\in \Delta_\ell^{\ell+1}$}.
                \end{cases}
$$
From the expression of $\nu$ in Theorem~\ref{th.measuretower}, we  deduce that
\begin{equation}
\nu\vert_{ \Delta_\ell^{\ell+1}}= \rho\nu_0\vert_{\{R=\ell+1\}}
\end{equation}
The previous considerations yield
\begingroup
\allowdisplaybreaks
\begin{align*}
\int_\Delta \log{J_T}^\pm(x,\ell) \, d\nu(x,\ell) &=\sum_{\ell=0}^\infty 
 \int_{\Delta_\ell}   \log{J_T}^\pm(x,\ell) \, d\nu(x,\ell)\\
&=\sum_{\ell=0}^\infty \int_{\Delta_\ell^{\ell+1}}   \log{J_F}^\pm(x) \, d\nu(x,\ell)\\
&= \frac1\rho\sum_{\ell=0}^\infty \int_{\{R=\ell+1\}}   \log{J_F}^\pm(x) \, d\nu_0(x)\\
& = \frac1\rho   \int_{\Delta_0}   \log{J_F}^\pm(x)  d \nu_0(x).
\end{align*}
\endgroup
The proof is complete.
\end{proof}

\begin{corollary}\label{co.integral}
 $\displaystyle\int_{\Delta} \log  {J_{T}}^{-} d\nu<\infty$.
\end{corollary}

\begin{proof}
Use  Corollary~\ref{co.integral0}  and Lemma~\ref{davids}.
\end{proof}
         
           The countable $m$ mod 0 partition $\mathcal{P}$ of $\Delta_0$ associated with the Gibbs-Markov  map $ F$ naturally induces an  $m$ mod $0$ countable partitions on each level of the tower. Collecting all these partitions, we obtain an $m$ mod $0$  partition  $\mathcal{Q}$ of the entire tower $\Delta$. A sequence of dynamically generated $m$ mod 0 partitions   of $\Delta$ is then defined,  for each $n\ge1$, by
           \begin{equation} \label{towerpartition}
\mathcal{Q}_n=\bigvee_{i=0}^{n-1}T^{-i} \mathcal{Q} .
\end{equation}
Given $(x,\ell)\in\Delta$, let $\mathcal Q_n(x,\ell)$ denote the element of $\mathcal Q_n$ containing the point $(x,\ell)\in\Delta$.
     For a proof of  the following bounded distortion property for $T$ see e.g.   \cite[Lemma 3.30]{A20}.

\begin{lemma}\label{le.jacT}
There exists   $C_1>0$ such that, for all $n\geq 0$ and $ (y,k) \in \mathcal Q_n(x,\ell)$,
\begin{equation*} \label{distor}
\frac1{C_1} \leq \frac{J_{T^n}(y,k)}{J_{T^n}(x,\ell)} \leq C_1. 
\end{equation*}
\end{lemma}


Our goal now  is to  compare   $m(\mathcal{Q}_n(x,\ell))$ with the Jacobian $J_{T^n}(x,\ell)$, at least for an infinite number of times $n\in \mathbb{N}$. This is the main technical difference between inducing schemes and tower extensions we have to deal with in order to replicate the same approach, since in the former case, owing to the Markov property~\ref{G1}, the comparison can be done for all times $n \in \mathbb{N}$. Such property does not pass down to the natural partiton $\mathcal{Q}$ of the tower but, fortunately, the strategy does not require all its strength. In the present situation, we use the approach in  \cite{AP21} for piecewise expanding maps, and the key feature of existence of infinitely many moments  for which the respective forward image of the refined atom   of a typical point has measure uniformly bounded away from zero. Here, the  bound is  given   by the measure of the ground level.
  Given  $(x,\ell) \in \Delta$, set 
$$
\mathcal{M}(x,\ell)=\left\{n\ge 1: m(T^{n}(\mathcal{Q}_n(x,\ell))) \geq m(\Delta_0)\right\}.
$$

\begin{lemma}\label{le.qusimarkov}
The set
$ 
\mathcal{M}(x,\ell)
$ 
has infinitely many elements, for $m$ almost every $(x,\ell)\in \Delta$.
\end{lemma}
\begin{proof}
The identification of $T^R$ and $F$   allows us to think of the   partitions in~\eqref{eq.novacoisa} as partitions of the base level of the tower.
%
We also use  $\mathcal P_n(x,0)$ to denote the element of $\mathcal P_n$ containing the point $(x,0)\in\Delta_0$.  
  Setting  for  each $n\ge 1$ 
 \begin{equation*}\label{eq.errek}
R_n=\sum_{j=0}^{n-1}R\circ (T^R)^j. 
\end{equation*} it easily follows that, for $m$ almost every $(x,0)\in\Delta_0$,
 $$
T^{R_n}(\mathcal{Q}_{R_n}(x,0)) = (T^R)^{ n}(\mathcal{P}_{ n}(x,0))=\Delta_0.
 $$
Recalling that $T$ is an upward translation between consecutive returns to the base, we have for $m$ almost all $(x,\ell)\in\Delta $
$$
 T^{R_n-\ell}(\mathcal{Q}_{R_n-\ell}(x,\ell)) = T^{R_n}(\mathcal{Q}_{R_n}(x,0)) =  \Delta_0.
 $$
 This clearly gives the conclusion.
\end{proof}

Note that the proof of the Lemma~\ref{le.qusimarkov} provides  the same conclusion for a more accurate  version of  $\mathcal{M}(x,\ell)$. In fact, it gives   $ T^{n}(\mathcal{Q}_n(x,\ell))=\Delta_0$, for infinitely many values of $n\in\mathbb N$. However,  a uniform lower bound for $m( T^{n}(\mathcal{Q}_n(x,\ell)))$ is really what we need  for the proof next lemma.


\begin{lemma}[Volume Lemma]
 \label{volume}
There exists   $C_0>0$ such that, for $m$ almost all $(x,\ell) \in\Delta $ and all  $n\in \mathcal{M}(x,\ell)$,  
$$
\frac1{C_0}\leq m(\mathcal{Q}_{n}(x,\ell))  J_{T^n}(x,\ell) \leq C_0.
$$
\end{lemma}

\begin{proof} 
%
For all $n\ge 1$ and $m$ almost all $(x,\ell)\in \Delta$, we may write
\begin{align}
m(T^n(\mathcal{Q}_{n}(x,\ell))) &= \int_{\mathcal{Q}_{n}(x,\ell)} J_{T^n}(y,k) \, d m(y,k)\nonumber\\
 &=\int_{\mathcal{Q}_{n}(x,\ell)} \frac{J_{T^n}(y,k)}{J_{T^n}(x,\ell)} J_{T^n}(x,\ell) \,  d m(y,k).\label{eq.also}
\end{align}
It follows  from~\eqref{eq.also} and~Lemma~\ref{le.jacT} that
$$
m(\Delta) \geq m(T^n(\mathcal{Q}_{n}(x,\ell)))\geq C_1^{-1}  J_{T^n}(x,\ell)  m(\mathcal{Q}_{n}(x,\ell)),
$$
and consequently, 
$$
J_{T^n}(x,\ell)  m(\mathcal{Q}_{n}(x,\ell)) \leq   C_1 m(\Delta) .
$$
In addition, it  follows from \eqref{eq.also},~Lemma~\ref{le.jacT} and the definition of the set $\mathcal M(x,\ell)$ that, for all  $n \in \mathcal{M}(x,\ell)$,
$$
m(\Delta_0) \leq m(T^n(\mathcal{Q}_{n}(x,\ell))) \leq C_1  J_{T^n}(x,\ell)  m(\mathcal{Q}_{n}(x,\ell)),
$$
and therefore,
$$
 C_1^{-1}m(\Delta_0)  \leq J_{T^n}(x,\ell)   m(\mathcal{Q}_{n}(x,\ell)).
$$
\noindent
Taking $C_0=\max\left\{ m(\Delta) C_1, (C_1^{-1}m(\Delta_0) )^{-1}\right\}$, we are done.
\end{proof}

\section{Entropy  of  the tower system}

The goal of  this  section is to deduce  the entropy formula for the  tower system.
We could appeal to Rohlin's formula, which gives the entropy   transformations with a generating partition as the  integral of the Jacobian of the transformation with respect to the invariant measure; see e.g.  \cite[Theorem 9.7.3]{VO16}. However, to avoid introducing  new concepts, we provide in Proposition~\ref{entrinduce}  a  self-contained  proof of the  entropy formula  for the tower system. This   uses     a  general    version of Birkhoff Ergodic Theorem
for functions whose integral is defined; recall~\eqref{star2}.  
 The need  of this is due to the fact that    nothing  guarantees the integrability of $\log J_T$ with respect to $\nu$. The integrability of $\log {J_T}^-$  given by Corollary~\ref{co.integral} will be  sufficient for our purpose. 


\begin{lemma}[Generalised Birkhoff Theorem]
\label{birgen}
Let $\Phi:X\to X$ preserve an ergodic probability measure $\eta$ and $\varphi:X\to\mathbb R$ be such that $\int_X\varphi\, d\eta$ is defined. For $\eta$ almost every~$x\in X$,
$$\lim_{n \rightarrow \infty} \frac{1}{n} \sum_{i=0}^{n-1} \varphi(\Phi^i(x)) = \int_X\varphi d\eta.$$
\end{lemma}

\begin{proof} 
If $\varphi\in L^1(\eta)$,    the conclusion  is just the classical case of Birkhoff Ergodic Theorem.  Otherwise,  one of the functions $\varphi^\pm$ belongs in $L^1(\eta)$ and the  integral of the other one is equal to $+\infty$.
Suppose  for definiteness that $\int\varphi^+d\eta=+\infty$. In such  case, we have $\int\varphi d\eta=+\infty$.
Set for each $n\ge 1$
$$S_n\varphi =\sum_{j=0}^{n-1}\varphi\circ\Phi^j.$$
Note that both $S_n\varphi$ and   $\int\varphi d\eta$ are linear in $\varphi$. Since   $\varphi=\varphi^+-\varphi^-$,   it is enough to show that, for $\eta$ almost every $x\in X$, 
 \begin{equation}\label{eq.fracinfty}
\lim_{n\to\infty}\frac1nS_n\varphi^+(x)=+\infty.
\end{equation}
Set for each $N\ge 1$, 
 $$
 \varphi_N=\min\{N,\varphi^+\}.
 $$
By the Monotone Convergence Theorem, 
  \begin{equation}\label{eq.falta}
\lim_{N\to\infty}\int\varphi_Nd\eta= \int\varphi^+d\eta=+\infty.
\end{equation}
 Also, for all $x\in X$ and $n,N\ge 1$,
\begin{equation}\label{eq.barulho}
 \frac1n S_n\varphi^+(x)\ge  \frac1n S_n\varphi_N(x).
\end{equation}
 Since $\eta$ is a finite measure and $\varphi_N$ is bounded, we have $\varphi_N\in L^1(\eta)$, for all $N\ge1$.  
 It follows from Birkhoff Ergodic Theorem that, for $\eta$ almost every $x\in X$,
$$ \lim_{n\to\infty}\frac1n S_n\varphi_N(x)=\int\varphi_Nd\eta ,$$
which together with~\eqref{eq.falta} and~\eqref{eq.barulho} yields~\eqref{eq.fracinfty}.
\end{proof}

In Lemma~\ref{gener.} below,  we show   that the entropy of the natural partition~$\mathcal Q$ introduced in Section~\ref{jacandnat} gives the entropy  of the tower system.  Recall that 
  \begin{equation}\label{eq.htQ}
h_{\nu}(T,\mathcal{Q})=\inf_{n\ge1}\frac1n H_\nu(\mathcal Q_n)
\end{equation}
and
$$H_\nu(\mathcal Q_n) =- \sum_{ \omega\in\mathcal Q_n}\nu( \omega)\log\nu( \omega).$$

\begin{lemma} \label{gener.}
$h_{\nu}(T)=h_{\nu}(T,\mathcal{Q})$.
\end{lemma}

\begin{proof}
By  the separation property \ref{G4} for the Gibbs-Markov map $T^R:\Delta_0\to\Delta_0$, we have that
 $  \mathcal P_\infty$ is the partition into single points of $\Delta_0$.  This   implies    that
 $  \mathcal Q_\infty$ is the partition into single points of $\Delta$. 
  Since $\mathcal Q_1 \le \mathcal Q_2\le\cdots$, it follows from  \cite[Corollary 5.12]{P69a} that
 $$h_\nu(T)=\lim_{n\to\infty} h_\nu(T,\mathcal Q_n).$$
 On the other hand, it is a general fact that $h_\nu(T,\mathcal Q)=h_\nu(T,\mathcal Q_n)$, for all $n\ge1$. 
This completes the proof.
\end{proof}

In the next result, we establish the entropy formula for tower systems for which the natural partition of  the tower has finite entropy. It follows from Lemma~\ref{gener.} and~\eqref{eq.htQ} that the finiteness of $H_\nu(\mathcal Q)$  is actually equivalent to the finiteness of the entropy of the tower system. The example in Section~\ref{se.systemno} shows that the conclusion   is no longer valid if the entropy is not finite.

\begin{proposition} \label{entrinduce}
If   $H_\nu(\mathcal Q)<\infty$, then 
$$\displaystyle h_{\nu}(T)=\int_\Delta   \log J_T  d\nu.$$
\end{proposition}

\begin{proof} First of all, note that, as we assume $H_\nu(\mathcal Q)<\infty$,   Shannon-McMillan-Breiman Theorem can be used  to obtain~$h_{\nu}(T,\mathcal{Q})$.  Taking  $(x,\ell)\in\Delta$ a    generic point with respect to $\nu$ (or $m$), the proof is produced with the following ingredients:

\begin{enumerate}
\item   $\mathcal Q$ is a generating partition (Lemma \ref{gener.});
\item Shannon-McMillan-Breiman Theorem;
\item $d\nu/dm$ is bounded from above and below by positive constants (Theorem~\ref{th.measuretower});
\item $\mathcal{M}(x,\ell)$ has infinitely many elements (Lemma~\ref{le.qusimarkov});
\item Volume Lemma (Lemma~\ref{volume});
\item Chain Rule for the Jacobian;
\item Generalised Birkhoff Theorem for $\log J_{T}$ (Corollary~\ref{co.integral} \& Lemma~\ref{birgen});
\end{enumerate}
and prepared as follows:
\begingroup
\allowdisplaybreaks
\begin{equation*}
\begin{split}
h_{\nu}(T) &\stackrel{(1)}=h_{\nu}(T,\mathcal{Q})\\
&\stackrel{(2)}=\lim_{n\rightarrow \infty} -\frac{1}{n} \log \nu(\mathcal{Q}_n(x,\ell))\\
&\stackrel{(3)}=\lim_{n\rightarrow \infty}-\frac{1}{n} \log m(\mathcal{Q}_n(x,\ell))\\
	&\stackrel{(4)}=  \lim_{\substack{n \to \infty  \\ n \in \mathcal{M}(x,\ell)}} -\frac{1}{n} \log m(\mathcal{Q}_{n}(x,\ell))\\
	&\stackrel{(5)}= \lim_{\substack{n \rightarrow \infty  \\ n \in \mathcal{M}(x,\ell)}} \frac{1}{n} \log J_{T^{n}}(x,\ell) \\	
	&\stackrel{(6)}=  \lim_{\substack{n \rightarrow \infty  \\ n \in \mathcal{M}(x,\ell)}} \frac{1}{n} \sum_{i=0}^{n-1} \log J_T(T^i(x,\ell)) \\
		&\stackrel{(7)}= \int_\Delta  \log J_T \,\, d\nu.
\end{split}
\end{equation*}
\endgroup
This gives the conclusion.
\end{proof}

We finish this section with a result that will be used  to obtain  the equivalence in  the second  item of 
  Theorem~\ref{th.formendo}.

\begin{lemma}\label{le.giniteq}
$H_{\nu}(\mathcal Q)<\infty$ if, and only if, $\displaystyle\int_{\Delta_0} R\log J_{F}dm<\infty$.
\end{lemma}

\begin{proof}
Set for each $n\ge 1$
$$\mathcal R_n= \left\{\omega\in\mathcal P\colon R(\omega) =n\right\}.
$$
By definition of the tower, for each $\omega\in\mathcal R_n$, there are exactly $n$ elements $\tilde\omega\in\mathcal Q$ above $\omega$  and, by Theorem~\ref{th.measuretower}, 
$$\nu(\tilde\omega)=\frac1\rho \nu_0(\omega),$$
with $\rho$ as in~\eqref{eq.rho}.
Hence,
\begin{eqnarray}
H_\nu(\mathcal Q) &=&- \sum_{\tilde\omega\in\mathcal Q}\nu(\tilde\omega)\log\nu(\tilde\omega)\nonumber\\
&=&\frac1\rho\sum_{n\ge 1} \sum_{\omega\in\mathcal R_n}n  \nu_0(\omega)\log\left(\frac\rho {\nu_0(\omega)}\right)\nonumber\\
&=&\frac1\rho \int_{\Delta_0} R(x) \log \left(\frac\rho{  \nu_0(\mathcal P(x) )}\right) d\nu_0(x).
\label{eq.nomn}
\end{eqnarray}
Using Lemma~\ref{le.lower}, we get
 \begin{equation*}\label{eq.gsd0}
\frac\rho{C}J_F(x)\le \frac\rho{ \nu_0(\mathcal P(x) )}\le C\rho J_F(x),
\end{equation*}
and so
 \begin{equation}\label{eq.gsd}
\log\frac\rho{C}+\log J_F(x)\le  \log \left(\frac\rho{  \nu_0(\mathcal P(x) )}\right)\le \log (C\rho) +\log J_F(x).
\end{equation}
It follows from~\eqref{eq.nomn} and~\eqref{eq.gsd} that
$$ \log \frac\rho{C}\int_{\Delta_0} Rd\nu_0+\int_{\Delta_0} R\log J_F d\nu_0 \le  \rho H_\nu(\mathcal Q) \le 
\log({C}\rho)\int_{\Delta_0} Rd\nu_0+\int_{\Delta_0} R\log J_Fd\nu_0.
$$
Since  $d\nu_0/dm$ is bounded from above and below by positive constants and $\int_{\Delta_0} Rd\nu_0<\infty$, we get the conclusion. 
\end{proof}

\section{Entropy of the original system}\label{se.original}

Here, we complete the proof of Theorem~\ref{th.formendo}.  Since
the tower system $(T,\nu)$ is an extension of  $(f,\mu)$ with countably many fibers, we may  use \cite[Proposition 2.8]{B99b} to get
\begin{equation} \label{buzzi28}
h_\mu(f)= h_\nu(T).
\end{equation}
%
%
By Lemma~\ref{gener.} and~\eqref{eq.htQ}, we have
\begin{equation}\label{eq.final}
 h_{\nu}(T)=h_{\nu}(T,\mathcal{Q})=\inf_{n\ge1}\frac1n H_\nu(\mathcal Q_n).
\end{equation}
Since $H_\nu(\mathcal Q)\le H_\nu(\mathcal Q_n)$, for all $n\ge1$, it follows from~\eqref{eq.final} and   Lemma~\ref{le.giniteq} that
\begin{equation}\label{eq.final2}  h_{\nu}(T)<\infty \iff H_\nu(\mathcal Q) <\infty \iff \int_{\Delta_0} R\log J_{F}dm<\infty,
\end{equation}
which together with~\eqref{buzzi28} yields 
\begin{equation*} 
h_\mu(f)<\infty \iff \int_{\Delta_0} R\log J_{F}dm<\infty.
\end{equation*}
This gives  the second  item of 
  Theorem~\ref{th.formendo}.
The first item 
   is    a  consequence of~\eqref{buzzi28},~\eqref{eq.final2},  Proposition~\ref{entrinduce} and  the next result.  
  

\begin{lemma} \label{vvvvv}   
$\displaystyle\int_\Delta  \log J_T \,d\nu = \int_M   \log |\det Df|\,  d \mu$.
\end{lemma}
\begin{proof} 
Consider the measurable partition 
$\mathcal{Q}_R=\{\Delta_\ell^n\}$ of $\Delta$, defined for all $ \ell\in \mathbb{N}_0$ and $ n>\ell$ by
$$
\Delta_\ell^n=\{(x,\ell) \in \Delta \colon  R(x)=n \}.
$$ 
Clearly, the natural partition $\mathcal{Q}$ is a refinement of $\mathcal{Q}_R$. Recalling~\eqref{eq.jacob}, we have  
$$
J_T|_{\Delta_\ell^n}(x,\ell)=
                \begin{cases}
                  1,  & \mbox{if $n>\ell+1$};
                 \\
                 J_{f^R}(x), & \mbox{if $n=\ell+1$}.
                \end{cases}
$$
Writing for simplicity  $J_f=|\det Df|$ and  using  the chain rule, we have for each  $x \in \Delta_0$ with $R(x)=n$ 
$$
J_{f^R}(x)=J_f(f^{n-1}(x))\cdots J_f(f(x)) \cdot J_f(x).
$$ 
The previous considerations 
yield
\begingroup
\allowdisplaybreaks
\begin{align*}
\int_\Delta \log J_T(x,\ell) \, d\nu(x,\ell) &=\sum_{\ell=0}^\infty 
 \int_{\Delta_\ell}   \log J_T(x,\ell) \, d\nu(x,\ell)\\
&=\sum_{\ell=0}^\infty \int_{\Delta_\ell^{\ell+1}}   \log J_{f^R}(x) \, d\nu(x,\ell)\\
	&=\sum_{\ell=0}^\infty \sum_{i=0}^\ell \int_{\Delta_\ell^{\ell+1}} \log J_f( f^i(x)) \, d\nu (x,\ell)\\
	&= \sum_{\ell=0}^\infty \sum_{n>\ell} \int_{\Delta_\ell^n}   \log J_f ( f^\ell(x)) \, d\nu(x,\ell) \\	
	&= \sum_{\ell=0}^\infty \int_{\Delta_\ell}  \log J_f ( f^\ell(x)) \, d\nu(x,\ell) \\
		&= \sum_{\ell=0}^\infty \int_{\Delta_\ell}   \log J_f ( \pi(x,\ell)) \, d\nu(x,\ell) \\
		&= \int_{\Delta}  \log J_f (\pi (x,\ell))\, d\nu(x,\ell) .
\end{align*}
\endgroup
Finally, observing that $\pi_*\nu=\mu$, we get
$$
\int_{\Delta} \, \log J_f\circ \pi \, d\nu =\int_{M} \, \log J_f \, d\pi_*\nu=\int_{M} \, \log J_f \, d\mu.
$$
The proof is complete.
\end{proof}

\section{Infinite entropy}\label{se.systemno}

In this section, we give  an  example of a piecewise smooth interval map  having  an  SRB measure given by   a Gibbs-Markov induced map and not satisfying  the entropy formula.
To build this  example, we are actually going to give
\begin{itemize}
\item[\mylabel{i1}{($i_1$)}] an interval  $\Delta_0$ with finite Lebesgue measure $m$;
\item[\mylabel{i2}{($i_2$)}]   a  piecewise $C^\infty$  Gibbs-Markov map $F:\Delta_0\to\Delta_0$ with   associated partition $\mathcal P$;
\item[\mylabel{i3}{($i_3$)}]   a   measurable function $R:\Delta_0 \to \mathbb N$   constant on each element of $\mathcal P$, with $R\in L^1(m)$. 
\end{itemize}
The tower map $T:\Delta\to\Delta$ associated with these objects   can  easily  be thought of   as~a piecewise $C^\infty$   map of an interval with countably many smoothness domains and  
$$J_T =|T '|$$ on each smoothness domain. In this case,  the return to the base of the tower $T^R:\Delta_0\to\Delta_0$, which  is naturally identified with the map $F$, is a Gibbs-Markov induced map for $T$.   Therefore, we need   to present objects as in \ref{i1}-\ref{i3} such that for  the corresponding tower map $T:\Delta\to\Delta$ we have 
\begin{equation}\label{notentro}
h_\nu(T)=\infty\qand \int_\Delta  \log J_T \, d{\nu}<\infty,
\end{equation}
where $\nu$ is  the unique  $T$-invariant probability measure such that $\nu\ll m$, given by Theorem~\ref{th.measuretower}. 
We are going to use
the continuous map $\phi:[0,1/e]\to[0,1/e]$, given by  $\phi(0)=0$ and 
 $$\phi(x)=-x\log x,$$
 for each $x>0$. 
Note that, for all $0<x<1/e$,
\begin{equation}\label{eq.deriva}
\phi'(x)=-\log x-1 >0\qand \phi''(x)=-\frac1x<0. 
\end{equation}
Hence, 
$\phi$ is an increasing concave function. Since $\phi$ is continuous at $0$ and $\phi(1/e)=1/e$, then $\phi$ is a  bijection.
For each $n\ge2$, set
 $$
 a_n=\phi^{-1}\left(\frac1{n^2 \log n}\right).
 $$
  and
$$ b_n=\sum_{i=2}^{n}a_i .$$ 
Set also $b_1=0$. In Lemma~\ref{le.sums} below   we show   that $$ b=\sum_{n=2}^{\infty}a_n<+\infty.$$
 Consider 
  $\Delta_0=\left[0, b\right]
$
and  the $m$ mod 0 partition $\mathcal P=\{\omega_n\}_{n\ge 2}$ of $\Delta_0$, where
 $$\omega_n=(b_{n-1},b_n),
 $$
 for each $n\ge 2$.
We define $F:\Delta_0\to\Delta_0$, mapping  each $\omega_n$ linearly onto $\Delta_0$. 
 In this way, we have for each $n\ge2$
  $$J_F\vert_{\omega_n}=|F'\vert_{\omega_n}|=\frac{b}{a_n}\ge \frac{b}{a_2}>1.$$
 It is not difficult to check  that $F$ is a Gibbs-Markov map; see e.g. \cite[Lemma 3.3]{A20}. We also define a map $R:\Delta_0\to\mathbb N$, setting for all $n\ge 2$
  $$R\vert_{\omega_n}=n.$$
Finally, consider the tower map $T:\Delta\to\Delta$ as in Section~\ref{towerextension} associated with these objects.
Observe that, for each $n\ge2$,
 $$m(\{ R=n\})= m(\omega_n)=a_n.$$
It follows from the construction   above  that
 \begin{align}
\int_{\Delta_0}R \, dm &=\sum_{n\ge2}n a_n,\label{um}\\
  \int_{\Delta_0}\log J_F dm &=\sum_{n\ge2}  a_n (\log b-\log a_n) =\log b \sum_{n\ge2}  a_n + \sum_{n\ge2}\phi(a_n) ,\label{dois}\\
   \int_{\Delta_0}R\log J_F dm &=\sum_{n\ge2}  n a_n (\log b-\log a_n) =\log b \sum_{n\ge2} n a_n + \sum_{n\ge2}n\phi(a_n) .\label{tres}
\end{align}
In Lemma~\ref{le.sums} below, we show that $\sum_{n\ge2}n a_n<+\infty$, which then gives  $R\in L^1(m)$.
By Lemma~\ref{davids}, 
$$\displaystyle\int_\Delta  \log J_T \, d{\nu} =\frac1\rho  \int_{\Delta_0}   \log J_F \,  d \nu_0,$$
with $
\rho= \int R d\nu_0<\infty$ as in~\eqref{eq.rho}.
 Recalling the second item of Theorem~\ref{th.formendo} and \eqref{dois}-\eqref{tres} above, 
 we   obtain~\eqref{notentro} from the next result.

\begin{lemma}\label{le.sums} \ 
\begin{enumerate}
\item[\mylabel{s1}{(1)}] $\sum_{n\ge2}a_n<\infty$;
\item[\mylabel{s2}{(2)}]  $\sum_{n\ge2}n a_n<\infty$;
\item[\mylabel{s3}{(3)}]  $\sum_{n\ge2}\phi(a_n)<\infty$;
\item[\mylabel{s4}{(4)}]  $\sum_{n\ge2}n \phi(a_n)=\infty$.
\end{enumerate}
\end{lemma}

\begin{proof}

Note that 
$\phi(a_n)=1/({n^2 \log n}),$
for each $n\ge2$.
Therefore, 
\ref{s3} is obvious and~\ref{s4} can be easily obtained  by  the integral test. Note also that \ref{s2} implies \ref{s1}. It remains  to verify the second item. 
For each $n\ge2$,
 $$  a_n= \phi^{-1}\left(\frac1{n^2\log n}\right)=\phi^{-1}\left(\frac1{n^2\log n}\right)-\phi^{-1}(0)=
 \frac{ ( \phi^{-1} )'(c_n)}{n^2\log n},
 $$
 for some $c_n$ between 0 and $1/(n^2\log n)$. It follows from~\eqref{eq.deriva} that, for each $n\ge2$,
 $$
  ( \phi^{-1} )'(c_n)=\frac{1}{-\log \phi^{-1}(c_n)-1}.
 $$
Using that $\phi^{-1}$ is an increasing concave function, we have for all $n\ge 3$
 $$
  0\le \phi^{-1}(c_n)\le  \phi^{-1}\left(\frac1{n^2\log n}\right)   \le  \frac1{n^2\log n} \le \frac1{n^2}.
 $$
   It follows that
   $$
  ( \phi^{-1} )'(c_n) 
  \le \frac{1}{2\log n-1} \le \frac{1}{ \log n },
 $$
  for all $n\ge3$,
 and so
 \begin{equation}\label{eq.enqua}
n a_n \le  \frac{1}{n \log^2 n}.
\end{equation}
Now, by the Cauchy condensation test,  
$$
 \sum_{n\ge2}\frac{1}{n \log^2 n} \le  \sum_{n\ge2}\frac{2^{n}}{2^n \log^2 2^n} =  \sum_{n\ge2}\frac{1}{n^2 \log^2 2}<\infty.
$$
Since this last series converges and \eqref{eq.enqua} holds for all $n\ge 3$, we get  the second item.
\end{proof}


\begin{remark}\label{re.smbt2}
The system $(T,\nu)$ is  a counterexample to Ruelle inequality in the setting of maps which are differentiable only almost everywhere. Indeed, we have for $\nu$ almost every $(x,\ell)\in\Delta$,
$$\lim_{n\to\infty}\frac1n|(T^n)(x(x,\ell))|=\int_{\Delta}\log |T'| d\nu<\infty= h_\nu(T).
$$
This is  the reverse of Ruelle inequality. 
\end{remark}

Regarding the previous remark, it will be interesting to  note that  Ruelle inequality is  obtained in \cite[Proposition 1]{BSV18} for Lipschitz maps. Actually, the proof uses only that the map is   uniformly continuous and differentiable almost everywhere, the latter  being assured by  Rademacher Theorem. The system $(T,\nu)$  also illustrates how important is the uniform continuity for the conclusion of \cite[Proposition 1]{BSV18}.

\begin{remark}\label{re.smbt}
The system $(T,\nu)$  illustrates  that  the conclusion of  Shannon-McMillan-Breiman Theorem does not hold in case  we do not assume that the generating partition has finite entropy. Indeed,   in the previous sections, condition $H_\nu(\mathcal Q)<\infty$ (which is equivalent to $h_\nu(T)<\infty$) is used  only in step (2) of the proof of Proposition~\ref{entrinduce}. Therefore, equalities (3)-(7) in that proof remain valid for $(T,\nu)$, yielding for every  generic point $(x,\ell)\in\Delta$
$$
\lim_{n\rightarrow \infty}-\frac{1}{n} \log \nu(\mathcal{Q}_n(x,\ell))=\int_{\Delta}\log |T'| d\nu <\infty = h_\nu(T).
$$
\end{remark}

 \section{Applications}\label{se.app1}

In the  subsections below, we apply Theorem~\ref{th.formendo} and Corollary~\ref{co.entform0} to some classes of piecewise smooth maps having  SRB measures given by Gibbs-Markov induced maps. We admit the possibility   that in some of the cases  the entropy formula has  already be obtained by other methods. 
Anyway, we include them all here as a way of illustrating how vast and general the applications of our main results in this first part can be. We present  some classes of  concrete examples, but we could have placed more abstract classes such as those dealt with  in 
\cite{DHL06} or \cite{E20}.

\subsection{Lorenz-like  maps}\label{se.singular}

Here, we consider a  class of piecewise expanding one-dimensional maps with a singular point,  introduced by Guckenheimer and Williams in \cite[Section 2]{GW79}. They  appear as the quotient of  a Poincaré return map  for  a geometric model  of the Lorenz flow; see    \cite[Section 3.3.2.1]{AP10a} for   details. 
%
Such quotient maps are  modeled by   a family of   transitive  maps $f: I\rightarrow I$  of the interval $I =[-1/2,1/2]$  satisfying the following properties:
\begin{enumerate} \label{prop 1dim L}
\item 
  $f$ is a $C^{1+}$ local diffeomorphism in  $ I \setminus \{0\}$ with 
$
f(0^+)=-1/2 $ and $ f(0^-)=1/2; 
$
\item 
there is $\lambda>1$ such that
$f'(x)\geq \lambda$, for all $x \in I \setminus \{0\}$;
\item 
  there is $0<\alpha<1$ such that 
$
  f'(x)\approx   |x|^{-\alpha}
$,
for all $x\in I\setminus\{0\}$.

\end{enumerate}
It is  well-known  that  such a map  $f$ admits a unique ergodic SRB measure $\mu$; see e.g. \cite[Corollary 3.4]{V97a}.  Moreover, it was  proved in \cite[Theorem 1]{DHL06} that $\mu$ is given by  a Gibbs-Markov induced map. 
It is easily verified that $\mathcal P=\{(-1/2, 0),( 0,1/2)\}$ is a $\mu$ mod~0 generating partition. Since  $\mathcal P$ is   finite, we have   $H_\mu(\mathcal P)<\infty$, and so  $h_\mu(f,\mathcal P)<\infty$. It follows that $h_\mu(f)<\infty$, and  Theorem~\ref{th.formendo} then yields
$$h_{\mu}(f)=\int_I  \log |f'| d\mu.
$$

\subsection{Rovella maps}

Considering a geometric construction similar to that   in \cite{GW79}, reversing only a relation in the eigenvalues of the singularity, Rovella  introduced in \cite{R93} the so-called \textit{contracting Lorenz attractor}   and managed to show that the one-dimensional quotient map associated with the Poincaré return map  has chaotic behaviour, no longer robustly, but persistently in a measure theoretical sense: it is given by  a   family  of maps  $f_a: I\to I$ of the interval  $I= [-1,1 ]$ such that  $f_a$ has a positive Lyapunov exponent for   $a$ belonging in a   set $E\subset [0,1]$ with~$0$ as a full density point, such that 
\begin{enumerate}

\item 
$f_0$ is a   local diffeomorphism on  $ I \setminus \{0\}$ with 
$
f_0(0^+)=-1 $ and $ f_0(0^-)=1; 
$

\item 
 $\max_{x>0} f_0'(x)=f_0'(1)$  and $\max_{x<0} f_0'(x)=f_0'(-1)$;

\item 
$ f_0$ has negative Schwarzian derivative in $I\setminus
\{0\}$;

\item 
$\pm 1$ are  pre-periodic  repelling points: there are  $k_1, k_2, n_1, n_2\ge 1$ such that
\begin{enumerate}
\item  $f_0^{n_1+k_1}(1)=f_0^{k_1}(1)$ and $  (f_0^{n_1})'(f_0^{k_1}(1))>1;$
\item   $f_0^{n_2+k_2}(-1)=f_0^{k_2}(-1) $ and $(f_0^{n_2})'(f_0^{k_2}(-1))>1.$
\end{enumerate}

\end{enumerate}
Moreover,  for all  $a\in [0,1]$  
\begin{enumerate}
\setcounter{enumi}{4}

\item 
 there is $0<\alpha<1$ such that 
$
  f_a'(x)\approx   |x|^{\alpha}
$,
for all $x\in I\setminus\{0\}$;
\item 
$ f_a$ is $C^3$ in $ I \setminus \{0\}$ with derivatives depending continuously on $a$;

\item 
the functions $ a\mapsto f_a(-1) $  and $ a\mapsto f_a(1) $ have derivative 1 at $a=0$.

\end{enumerate}
%
Note that  $x= 0$ is  a discontinuity point for each $f_a$ and also a critical point:  the side derivatives are both equal to zero.
It was proved in  \cite{M00} that  for all $a\in E$, the map $f_a$ has some ergodic SRB measure,  and  it was shown  in  \cite[Corollary B]{AS12}  that the   SRB measure  is unique.  Moreover, it follows from  \cite[Theorem A]{AS12} and 
\cite[Theorem 4.1]{A04}   that the SRB measure $\mu_a$   is given by an induced Gibbs-Markov map for which  $m(\{R>n\})$ decays exponentially fast to zero with $n$. This in particular implies that $\int R^2 dm<\infty$.  Since $f'$ is bounded, it follows from Corollary~\ref{co.entform0} that, for each $a\in E$,
$$h_{\mu_a}(f_a)=\int_I  \log |f_a'| d\mu_a.
$$

\subsection{Intermittent   maps}\label{se.intermittent}


In this subsection, we consider a class of maps which contain as a particular case the    model  studied   by Liverani, Saussol and Vaienti in \cite{LSV99}. 
Given $ \alpha>0$, consider   the interval $I=[0,1]$, a point $ z_0\in(0,1)$  and a map $f:I\to I$, defined by
\begin{equation}\label{eq.LSVgen}
 f_\alpha (x)=
 \begin{cases}
 g_0(x), &\text {if $0\le x\le z_0$;}\\
 g_1(x), &\text {if $z_0< x\le1 $.}
 \end{cases}
\end{equation}
Assume  $g_1$ is a $C^2$ map with  derivative    strictly greater than one    mapping  $(z_0,1]$ diffeomorphically to $(0,1]$, 
and $g_0$ is  such that       
\begin{enumerate}
 \item 
    $f_\alpha(0)=0$ and $f_\alpha(z_0)=1$;
   \item 
   $f'_\alpha(0)=1$ and $f'_\alpha(x)>1$ for  $x\in (0,z_0]$;
 \item 
   $f_\alpha$ is $C^2$ on $(0,z_0]$  
  and 
  $f''_\alpha(x)\approx x^{\alpha-1}.$
\end{enumerate}
Note that   $f_\alpha$ is      $C^{1+\alpha}$   on $I\setminus\{z_0\}$, for each $0<\alpha<1$, with a discontinuity  at~$z_0$. 
The  model  in \cite{LSV99} is given by the particular choices  of 
 \begin{equation}\label{eq.LSV}
 z_0=\frac12,\quad g_0(x)=x +2^\alpha x^{\alpha+1}\qand g_1(x)=2x-1.
 \end{equation}
For $\alpha\ge 1$, the Dirac measure at zero is a physical measure for $f_\alpha$ and its basin covers $\leb$ almost all of $I$; see e.g. \cite[Theorem 3.59]{A20}. Therefore,  $f_\alpha$ has no SRB measure for $\alpha\ge 1$. 
For $0<\alpha<1$, the map $f_\alpha$ has a unique    SRB measure $\mu_\alpha$ given by a Gibbs-Markov induced map; see e.g. \cite[Theorem 3.59]{A20}. It is easily verified that $\mathcal P=\{(0,z_0),(z_0,1)\}$ is a $\mu_\alpha$ mod~0 generating partition. Since  $\mathcal P$ is   finite, we have   $H_{\mu_\alpha}(\mathcal P)<\infty$, and so  $h_{\mu_\alpha}(f_\alpha,\mathcal P)<\infty$. It follows that $h_{\mu_\alpha}(f_\alpha)<\infty$. Theorem~\ref{th.formendo} then yields
$$h_{\mu_\alpha}(f_\alpha)=\int_I  \log |f'_\alpha| d\mu_\alpha.
$$

\subsection{Singular intermittent maps}
Here, we consider a  family of     maps $f: S^1\to S^1$ of the circle $S^1=[-1,1]/\!\sim$, with $-1\sim 1$, introduced   by Cristadoro, Haydn, Marie and Vaienti  in~\cite{CHM10}, combining  the infinite derivative in  Section~\ref{se.singular} with the intermittency phenomenon in Section~\ref{se.intermittent}. The image  $f(x)$  depends on a parameter $\gamma>1$ and is defined implicitly for $x>0$ by 
$$
x=
\begin{cases}
\displaystyle
\frac1{2\gamma}(1+f(x)^\gamma, & \text{if } 0\le x\le \displaystyle\frac1{2\gamma};\\
f(x)+  \displaystyle\frac1{2\gamma}(1-f(x)^\gamma, &\text{if }  \displaystyle\frac1{2\gamma}< x\le 1.
\end{cases}
$$
and for $x<0$ by $f(x)=-f(-x)$. The  map $f$  is   $C^1$ on $S^1\setminus \{0\}$ and $C^2$ on $S^1\setminus \{0, 1\}$, with 
$$\lim_{x\to 0^\pm} f'(x)=+\infty.$$
  For the limit case    $\gamma=1$, this is the well-known doubling map, and for the particular value  $\gamma=2$, this is the map considered in \cite{AA04} as an example of a system with positive frequency of hyperbolic times but non-integrability of the first hyperbolic time. As observed in~\cite{CHM10}, the Lebesgue measure $m$  is $f$-invariant for all $\gamma>1$. Moreover, it is given  by an induced  Gibbs-Markov map; see~\cite[Section 3]{CHM10}.
It is easily verified that $\mathcal P=\{(-1, 0),( 0,1)\}$ is an $m$ mod~0 generating partition. Since  $\mathcal P$ is   finite, we have   $H_m(\mathcal P)<\infty$, and so  $h_m(f,\mathcal P)<\infty$. This implies that $h_m(f)<\infty$.   Theorem~\ref{th.formendo} then yields
$$h_{m}(f)=\int_{S^1}  \log |f'| dm.
$$

 \subsection{Skew-product intermittent maps}
In this subsection, we consider a family of two-dimensional piecewise smooth maps introduced by Bahsoun, Bose and Duan in \cite{BBD14}, based on a previous skew-product random map in~\cite{BBQ08}. Consider the interval $I=[0,1]$ and, for each $\alpha>0$, let $f_{\alpha}:I\to I$ be  defined   as in~\eqref{eq.LSVgen} with the particular choices of $z_0,g_0,g_1$ as  in~\eqref{eq.LSV}. 
Take real numbers $\alpha_0,\alpha_1,p_0,p_1>0$ such that $0<\alpha_0<\alpha_1\le 1$ and $p_0+p_1=1$.  
Consider  the $C^{1+\alpha_0}$ skew-product transformation 
$f:I\times I\to I\times I$ by 
$$
f(x,y)=( f_{\alpha(y)}(x),  \varphi(y)),
$$
where
$$
\alpha(y)=
\begin{cases}
\alpha_0, &\text{if $y\in[0,p_0)$};\\
\alpha_1, &\text{if $y\in[p_0,1]$;}
\end{cases}
\qand
\varphi(y)=
\begin{cases}
\dfrac{y}{p_1}, &\text{if $y\in[0,p_0)$};\\
\dfrac{y-p_0}{p_1}, &\text{if $y\in[p_0,1]$.}
\end{cases}
$$
Note that the discontinuity points of $f$ are given by  the lines $x=1/2$ and $y=p_0$. Therefore, $f$ has a partition into four smoothness domains. It follows from the results in~\cite[Section~3]{BBD14} that $f$ has a unique    SRB measure $\mu$ given by a Gibbs-Markov induced map. It is easily verified that 
$\mathcal P=\{(0,1/2)\times (0,p_0),(0,1/2)\times (p_0,1), ( 1/2,1)\times (0,p_0),( 1/2,1)\times (p_0,1)\}$
 is a $\mu $ mod~0 generating partition. Since  $\mathcal P$ is   finite, we have   $H_{\mu }(\mathcal P)<\infty$, and so  $h_{\mu }(f ,\mathcal P)<\infty$. It follows that $h_{\mu }(f )<\infty$. Theorem~\ref{th.formendo} then yields
$$h_{\mu }(f )=\int_{I\times I}  \log |f' | d\mu .
$$


\part{Systems with hyperbolic structures}\label{part2}

Let $M$ be a finite
dimensional compact Riemannian manifold $M$. Let  $\dist$   be   the distance on $M$ and $\leb$ be  the Lebesgue (volume) measure on the Borel sets of~$M$    induced by the Riemannian metric. 
Given a submanifold $\gamma\subset M$,   we use $\dist_\gamma$ to denote the distance on $\gamma$ and 
$\leb_\gamma$ to denote the Lebesgue measure on $\gamma$  induced by the restriction of the 
 Riemannian metric to~$\gamma$. 
 Let $f:M\to M$ be a  $C^{1+\eta}$ piecewise diffeomorphism, meaning that there is a countable number of pairwise disjoint open regions $M_1,  M_2,\dots$ 
such that
  $\cup_{k\ge1} \overline{M}_k=M$ and   $f\vert_{\cup_{k\ge1} M_k}$ is a  $C^{1+\eta}$ diffeomorphism onto its image.
  We  refer to
   $$S= M\setminus \cup_{k\ge1}M_k$$
   as the \emph{singular set} of $f$. 
Typically, $S$ can be  a set of critical   points or points where the derivative of $f$ does not exist (possibly, discontinuity points).

\section{Young sets} \label{section.GMY}


We say that $\Gamma$  is a \emph{continuous family of $C^1$ disks} in~$M$, if 
 there are a compact metric space~$K$, a unit disk $D$ in some $\mathbb R^k$ and an injective continuous function
 $\Phi\colon K\times D\to M$ 
 such~that
 \begin{itemize}
\item $\Gamma=\left\{\Phi(\{x\}\times D )\colon x\in K\right\}$;
\item $\Phi$ maps $K\times D$ homeomorphically onto its image;
\item $x\mapsto \Phi\vert_{\{x\}\times D}$ defines a continuous map from $K$ into $\emb^1(D,M)$, where $\emb^1(D,M)$ denotes the space of $C^1$  embeddings of $D$ into $M$.
\end{itemize}
Note that the disks in $\Gamma$ have all the same dimension  (of the disk $ D$)  denoted by $\dim\Gamma$.
  We say that a compact set  $\Lambda\subset M$ has a \emph{product
structure}\index{product structure}, if there are continuous families  of $C^1$  Pesin stable disks $\Gamma^s$ and   unstable disks $\Gamma^u$ such that
\begin{itemize}
    \item $\Lambda=(\cup_{\gamma\in\Gamma^s}
    \gamma)\cap (\cup_{\gamma\in\Gamma^u}
    \gamma)$;
    \item $\dim \Gamma^s+\dim \Gamma^u=\dim M$;
    \item each  $\gamma\in \Gamma^s$ meets each $\gamma\in\Gamma^u$  in exactly one point.
\end{itemize}
We say that $\Lambda_0\subset \Lambda$ is  an
{\em $s$-subset}\index{$s$-subset}, if $\Lambda_0$ has a   product
structure with respect to families $\Gamma_0^s$ and $\Gamma_0^u$
such that  $\Gamma_0^s\subset\Gamma^s$ and~$\Gamma_0^u=\Gamma^u$; {\em $u$-subsets}\index{$u$-subset} are defined similarly. 
Let $\gamma^{*}(x)$ denote the disk in 
$\Gamma^{*}$ containing the point $x\in\Lambda$, for~$*=s,u$. 
Consider the \emph{holonomy map}\index{holonomy map}
   $\Theta_{\gamma,\gamma'}\colon\gamma\cap\Lambda\to\gamma'\cap\Lambda$, defined  for   each $x\in\gamma\cap\Lambda$ by
\begin{equation}\label{eq.teta1}
\Theta_{\gamma,\gamma'}(x)=\gamma^s(x)\cap \gamma'.
\end{equation}
 We say that a compact set~$\Lambda$ is a \emph{Young set}   
  if~$\Lambda$ has a   product structure given by continuous families of $C^1$ disks $\Gamma^s$ and $\Gamma^u$ such that 
  conditions \ref{Y1}-\ref{Y5} below are satisfied.
\begin{enumerate}
    \item[\mylabel{Y1}{(Y$_1$)}]  {\em Markov}:\index{Markov} there is a sequence $(\Lambda_i)_{i\ge 1}$ of pairwise disjoint $s$-subsets of $\Lambda$ 
 such that
    \begin{itemize}
 \item    $\leb_{\gamma} (\Lambda  \cap\gamma )>0$ and $\leb_{\gamma} (\Lambda\setminus\cup_{i\ge1}\Lambda_i)\cap\gamma )=0$, for all $\gamma\in\Gamma^u$;
 \item  for all $i\ge1$, there is $R_i\in\mathbb N$ such that $f^{R_i}(\Lambda_i)$ is a $u$-subset and,   for all $x\in \Lambda_i$,
         $$
         f^{R_i}(\gamma^s(x))\subset \gamma^s(f^{R_i}(x))\qand
         f^{R_i}(\gamma^u(x))\supset \gamma^u(f^{R_i}(x));
         $$
      \item  for all $i\ge1$,   $0\le j\le R_i$ and $x\in\Lambda_i$,
$$
f^j(\gamma^s(x))\cap S=\emptyset\qand f^j(\gamma^u(x))\cap S=\emptyset.
$$
    \end{itemize}

\end{enumerate}
%
This allows us to introduce the \emph{recurrence time}\index{recurrence time}    $R:\Lambda\to\mathbb N$ and  the \emph{return
 map}\index{return!map} $f^R:\Lambda\to\Lambda$ of the Young set $\Lambda$,  setting for  each
$i\ge1$
 \begin{equation}\label{defRfR}
 R\vert_{\Lambda_i}=R_i\qand
f^R\vert_{\Lambda_i}=f^{R_i}\vert_{\Lambda_i}. 
 \end{equation}
Note that  
$R$ and $f^R$ are  defined on  a full~$\leb_\gamma$ measure subset  of $\Lambda\cap \gamma$,  for each $\gamma\in\Gamma^u$.  
%
%
%
%
%
 Thus, there exists  a set $ \Lambda^\prime\subset\Lambda$ intersecting each $\gamma\in \Gamma^u$ in a full $m_\gamma$ measure subset,   such that    $(f^R)^n(x)$ belongs in some   $\Lambda_i$, for all $n\ge0$ and  $x\in\Lambda'$.
Given 
  $x,y\in\Lambda'$,  we define the
\emph{separation time}\index{separation!time}
   $$s(x,y)=\min\big\{  n\ge 0 :\,\text{$ (f^R )^n(x)$ and $ (f^R)^n(y)$ lie in distinct  $ \Lambda_i$'s}\big\},
      $$
   with the convention that $\min( \emptyset)=\infty$. 
For definiteness, we set  the separation time equal to zero for all other points. 
For the remaining conditions, we consider  constants~$C>0$ and $0<\beta<1$,  only depending on $f$ and $\Lambda$.

 \begin{enumerate}
\item[\mylabel{Y2}{(Y$_2$)}]   
  \emph{Contraction on stable disks}:    for all  $\gamma\in \Gamma^s$ and  $x,y\in\gamma$, 
\begin{itemize}
\item $\displaystyle\dist (f^n  (y), f^n (x) )\le C \beta^n$,\, for all $n\ge0$.
 \end{itemize}
 \end{enumerate}

 \begin{enumerate}
\item[\mylabel{Y3}{(Y$_3$)}]
  \emph{Expansion on unstable disks}:   for  all  $ i\ge1$,  $\gamma\in \Gamma^u$   and  $x,y\in\gamma\cap \Lambda_i$,
  \begin{itemize}
\item   $  \dist( (f^R)^n (y), (f^R)^n (x))\le C\beta^{s(x,y)-n}$,\, for all $n\ge0$; 
\item  $\displaystyle\dist (f^j(y),f^j(x) )\le C\dist(f^R(x),f^R(y))$,\, for all $1\le j\le R_i$.
  \end{itemize}
 \end{enumerate}

\begin{enumerate}
 \item[\mylabel{Y4}{(Y$_4$)}]
 \emph{Gibbs}:  \index{Gibbs} for all   $i\ge1$, $\gamma\in \Gamma^ u$   and  $x,y\in \gamma\cap\Lambda_i$,  
  \begin{itemize}
  \item $\displaystyle\log\frac{\det Df^{R}\vert T_x\gamma}{\det Df^{R}\vert T_y\gamma}\le
    C\beta^{s(f^{R}(x),f^{R}(y))}.$
    \end{itemize}
 \end{enumerate}
\begin{enumerate}
    \item[\mylabel{Y5}{(Y$_5$)}] \emph{Regularity of the stable holonomy}: \index{stable holonomy!regularity} 
    \index{absolutely continuous!foliation}   for all  $\gamma,\gamma'\in\Gamma^u$, the measure $(\Theta_{\gamma,\gamma'})_*m_\gamma$ is absolutely continuous with respect to $\leb_{\gamma'}$ and its density $\rho_{\gamma,\gamma'}$ satisfies 
\begin{itemize}
\item $\displaystyle
 \frac1C\le  
 \rho_{\gamma,\gamma'} 
 \le C $;
 \item
 $\displaystyle \log\frac{\rho_{\gamma,\gamma'}(x)}{\rho_{\gamma,\gamma'}(y)}\le C \beta^{s(x,y)},
 $
 for all $x,y\in\gamma'\cap\Lambda$.
 \end{itemize}
\end{enumerate}
We   say that a Young set   has \emph{integrable recurrence times}\index{recurrence time!integrable} if~$R$ is integrable with respect to the measure $\leb_{\gamma}$, for some $\gamma\in\Gamma^u$. And so, for all   $\gamma\in\Gamma^u$, by   \ref{Y5}. In the next two results, we see the utility of  Young sets (with integrable recurrence times) for obtaining SRB measures;
 see \cite[Section 2.2]{Y98} or \cite[Theorem 4.7 \& Theorem 4.9]{A20} for a proof. In the present  context, by an \emph{SRB measure}\index{measure!SRB}  we mean  an invariant probability measure whose conditionals on local unstable leaves are absolutely continuous with respect to the conditional Lebesgue measure on those leaves.

\begin{theorem}\label{th.SRB} 
Assume that $f^R:\Lambda\to \Lambda$ is the return map of a Young set with   integrable recurrence times. Then,
\begin{enumerate}
\item $f^R$ has a  unique ergodic  SRB measure $\mu_0$;
\item $f$ has a unique ergodic SRB measure $\mu$ with $\mu(\Lambda)>0$,  which is given by  
 $$ 
{\mu}=\frac1{\sum_{j= 0}^\infty\mu_0(\{  R > j\})} \sum_{j=0}^\infty f^j_{*}(\mu_0|\{  R > j\}).
$$ 
\end{enumerate}
   \end{theorem}

We will refer to  the measure $\mu$ as  the \emph{SRB measure given by the Young set $\Lambda$}.
It is not difficult to see that both $\mu_0$ and $\mu$ have $\dim \Gamma^s$ negative Lyapunov exponents and $\dim \Gamma^u$ positive Lyapunov exponents.
%
%
%
Bearing in mind the classical Oseledets Theorem, we define for $\mu$-almost (thus $\mu_0$-almost) every point $x\in M$ 
\begin{equation} \label{unstabjac}
Df_u(x)=  Df \vert_{E_x^{u}}\qand D{f^R_u}(x)=  Df^R |E_x^{u},
\end{equation}
where   $E_x^{u}$ is  the direct sum   of the subspaces in the Oseledets decomposition of $T_xM$ associated with the  positive Lyapunov exponents. Our main result  in this part   establishes in particular the classical Pesin entropy formula for  SRB measures given by Young sets; recall \cite[Proposition 2.5]{LS82}.

\begin{maintheorem} \label{entropy formula}
Let $f:M\to M$ be a   piecewise $C^{1+\eta}$ diffeomorphism 
with  an ergodic SRB measure~$\mu$ given by a Young set $\Lambda$ with recurrence time $R$. Then,
 \begin{enumerate}
\item  if $h_\mu(f)<\infty$, then
$$
h_\mu(f)=\int_M   \log |\det Df_u| \,  d\mu;
$$
\item $h_\mu(f)<\infty$ if, and only if, for some $\gamma_0\in\Gamma^u$,
  $$\int_{\gamma_0\cap\Lambda} R\,\log |\det D{f^R}\vert_{\gamma_0}| \,dm_{\gamma_0}<\infty . 
$$
\end{enumerate}
\end{maintheorem}

  The strategy for proving the first item of Theorem~\ref{entropy formula} is slightly more involving than that for   Theorem~\ref{th.formendo}. We consider a tower extension $(\widehat T,\hat\nu)$  of $(f,\mu)$, but also a Gibbs-Markov quotient map  $F$ of $f^R$ with an invariant measure $\nu_0$ and  the tower system $(T,\nu)$ associated with~$(F,\nu_0)$, as is Section~\ref{towerextension}. We also consider  the natural extensions $(\widehat T^\#,\hat\nu^\#)$ and $(T^\#,\nu^\#)$ of $(\widehat T,\hat\nu)$ and $(T,\nu)$, respectively. Finally, we prove that, taking  $\rho>0$ as  in~\eqref{eq.rho}, we have  
\begin{align*}\label{eq.entros}
h_\mu&(f)=h_{\hat{\nu}}(\widehat{T})=h_{\hat{\nu}^{\#}}(\widehat{T}^{\#})=h_{ {\nu}^{\#}}( {T}^{\#})=h_{{\nu}}({T})\\
& =\int_\Delta  \log J_T \, d{\nu} =\frac1\rho   \int_{\gamma_0\cap\Lambda}   \log J_F \,  d \nu_0=\frac1\rho \int_{\Lambda}  \log  |\det Df^R_u|  \,d  \mu_0= \int_M\log |\det Df_u|\,d\mu.
\end{align*}
The first equality   is   a consequence
of   a  general result due to Buzzi for extensions with countably many fibers. 
The remaining equalities in the first line are a consequence of  
a result due to Demers, Wright and Young 
establishing that the measure preserving systems $(\widehat T^\#,\hat\nu^\#)$ and $(T^\#,\nu^\#)$  are isomorphic, 
and 
a result due to Rohlin  which establishes  that the entropy of the natural extension is equal to the entropy of the original system. The four equalities in the second line 
will be obtained in 
Proposition \ref{entrinduce},  
  Lemma~\ref{davids},  Lemma~\ref{le.last1} and Lemma~\ref{le.igual}, respectively.
%
%

Assuming that there is  $C>0$ for which $|\det Df_u |\le C$, 
 it follows from the chain rule   that
  $ |\det D{f^R\vert_{\gamma_0}}|\le C^R.$ We therefore   have the following simple, albeit  useful, consequence of Theorem~\ref{entropy formula}.

\begin{maincorollary} \label{co.entform1}
Let $f:M\to M$ be a   piecewise $C^{1+\eta}$ diffeomorphism 
with  an ergodic SRB measure~$\mu$  given by a Young set $\Lambda$ with recurrence time $R$.  If $| \det Df_u|$ is bounded   and ${\int_{\gamma_0\cap \Lambda}R^2 dm_{\gamma_0}<\infty}$ for some $\gamma_0\in\Gamma^u$, then
$$
h_\mu(f)=\int_M   \log |\det Df_u| \,  d\mu<\infty.
$$
\end{maincorollary}

In Section~\ref{se.systemno2}, we provide   an example of a  piecewise~$C^\infty$diffeomorphism~$f$  with an SRB  measure~$\mu$ with infinite entropy given by a Young set for which
the   formula in the  first item of Theorem~\ref{entropy formula} is no longer valid.   In Section~\ref{se.app2}, we apply   Theorem~\ref{entropy formula} and Corollary~\ref{co.entform1} to some classes of piecewise  smooth  diffeomorphisms with nonempty singular sets. 

\section{Quotient return map}\label{sub.quotient}
In this section, we introduce a \textit{quotient map}   associated with the return map of a set  with a Young structure, by  collapsing  stable leaves.  Let   $\Lambda$ be a Young set  with return map  $f^R:\Lambda\to\Lambda$. 
Fixing some $\gamma_0\in\Gamma^u$, consider   $\Theta_{\gamma_0}:\Lambda\to \gamma_0\cap\Lambda$, setting  for each $x\in\Lambda$  
 \begin{equation}\label{eq.teta2}
 \Theta_{\gamma_0}(x)= \gamma^s(x)\cap \gamma_0. 
\end{equation} 
 We  define  the \emph{quotient map}\index{quotient!map} of   $f^R$  as 
%
%
\begin{equation}\label{eq.F01}
 \begin{array}{rcccl}
F&\colon & \gamma_0\cap \Lambda & \longrightarrow &
\gamma_0\cap \Lambda
 \\
& &x&\longmapsto & \Theta_{ \gamma_0}\circ f^R (x).
 \end{array}
 \end{equation}
 It is easily verified that  
  \begin{equation}\label{eq.semifRF}
F \circ\Theta_{\gamma_0} = \Theta_{\gamma_0}\circ f^R .
 \end{equation}
 To simplify notation,  the restriction of   $m_{\gamma_0 }$ to $\gamma_0\cap\Lambda$ will still be denoted by $m_{\gamma_0}$.
The proof of the next result is given in  \cite[Proposition~4.2]{A20}. 


\begin{proposition}\label{pr.F0Gibbs}
The quotient map $F:\gamma_0\cap \Lambda\to \gamma_0\cap \Lambda$ is Gibbs-Markov with respect to the $m_{\gamma_0}$ mod $0$ partition $\mathcal{P}=\{\gamma_0\cap \Lambda_1, \gamma_0\cap \Lambda_2, \dots \}$ of $\gamma_0 \cap \Lambda$. 
\end{proposition}

The  proof of  \cite[Proposition~4.2]{A20} gives that, for all  $\omega\in\mathcal P$,
\begin{equation}\label{eq.jacompoe}
J_{F}\vert_\omega=(\rho_{\gamma_1,\gamma_0}\circ f^R) \cdot |\det Df^{R}\vert_{\gamma_0}|, 
\end{equation}
with  $\gamma_1\in\Gamma^u$ such that $f^R(\omega)\subset\gamma_1$ and $\rho_{\gamma_1,\gamma_0}$ as in~\ref{Y5}.  
 By Theorem~\ref{th.SRB0},   the quotient map $F$ has a unique ergodic  invariant probability measure  absolutely continuous with respect to~$m_{\gamma_0}$. The next result relates this measure   to  the unique SRB measure given by Theorem~\ref{th.SRB} for the return map of  the Young set; see \cite[Lemma 4.5]{A20} for a proof.

\begin{proposition}\label{le.SRBproj} Let   $\Lambda$ be a Young set   and $F:\gamma_0\cap \Lambda\to \gamma_0\cap \Lambda$  be  a quotient  of the return map $f^R:\Lambda\to\Lambda$. 
If $\mu_0$ is the unique SRB measure for $f^R$, then
  $\nu_0=(\Theta_{\gamma_0})_*\mu_0$ is the unique  $F$-invariant probability measure such that  $\nu_0\ll\leb_{\gamma_0}$.
  \end{proposition}

This last result, together with~\eqref{eq.semifRF}, shows  that $(f^R,\mu_0)$ is an \emph{extension} of $(F,\nu_0)$.

\section{Tower extension}\label{hatower}


As in  Section~\ref{towerextension},  consider the \emph{tower}\index{tower} associated with $f^R:\Delta\to\Delta$,
 $$\hat\Delta=\big\{ (x,\ell)\colon \text{$x
\in \Lambda$ and $0\leq \ell < R(x)$} \big\},$$ 
 and the   \emph{tower map}\index{tower!map}  $\widehat T:\hat \Delta\to\hat\Delta$, given  by
 $$
\widehat T(x,\ell)= 
\begin{cases}
    (x,\ell+1), & \hbox{if $
\ell<R(x)-1$;} \\
    (f^R(x),0), & \hbox{if $\ell=R(x)-1$.} \\
\end{cases}%
 $$
As before, the base $\hat\Delta_0$ of the tower $\hat\Delta$  is naturally identified with the set    $\Lambda$, and each level~$\hat\Delta_\ell$ with the  set   $\{R>\ell\}\subset \Lambda$. This allows us to refer to stable and unstable disks through  points in the tower, naturally considering the corresponding disks of their representatives in the base. 
Also, the set $\Lambda$ is identified with the base level $\hat \Delta_0$ and the return map $f^R:\Lambda\to\Lambda$ is identified with the return to the base $T^R:\hat\Delta_0\to\hat\Delta_0$


We have seen in Theorem~\ref{th.SRB}  that the return map $f^R$ has a unique  SRB measure. In Theorem~\ref{th.SRBtower} below, we show that this SRB measure gives rise to a unique ergodic SRB measure for $\widehat T$.
Since each  level of the tower is  identified with a subset of the set with a product structure, it still makes sense to  talk about  SRB measures   for the  tower map $\widehat T$. 
Considering, as  in Section~\ref{towerextension}, the map
  \begin{equation}\label{eq.pi2}
 \begin{array}{rccl}
 \pi \colon &\!\!\!\!  \hat \Delta & \longrightarrow &
 M
 \\
& \!\!(x,\ell)&\longmapsto & f^\ell(x),
 \end{array}
 \end{equation}
we  also have
\begin{equation}\label{eq.cnjugar2}
f\circ \pi = \pi\circ \widehat T.
\end{equation}
The next result gives in particular that the push-forward under $\pi$ of the unique ergodic SRB measure for the tower map coincides with the ergodic SRB measure given by Theorem~\ref{th.SRB}.  A proof of this  result can be found in~\cite[Theorem 4.11]{A20}.

\begin{theorem}\label{th.SRBtower}
Let  $\widehat T$ the tower map associated with  the  return map $f^R$ of a Young  set~$\Lambda$ with  
with  integrable recurrence times. If $\mu_0$ is the  unique    SRB measure for~$f^R$, then 
$$ 
\hat{\nu}=\frac1{ \sum_{j=0}^{\infty}  \mu_0(\{  R > j\})} \sum_{j=0}^{\infty}\widehat T^j_{*}( \mu_0|\{  R > j\})
$$
is the  unique     SRB measure for  $\widehat T$. Moreover, $\hat\nu$ is ergodic and $\mu= \pi_*\hat\nu$ is the  unique    SRB measure for $f$ with $\mu(\Lambda)>0$. 
\end{theorem}

Together with~\eqref{eq.cnjugar2}, this means  that   the tower system $(\widehat T,\hat\nu)$ is an \emph{extension} of   $(f,\mu)$.

\section{Quotient tower}\label{sub.quotow}
Here,  we analyse the relationship between  the tower associated with the return map   and  the tower associated with  the quotient of the return map. 
Fix  some  $\gamma_0\in \Gamma^u $  and consider the quotient map   $$F:\gamma_0\cap \Lambda\to \gamma_0\cap \Lambda,$$ as in~\eqref{eq.F01}. 
By  Proposition~\ref{pr.F0Gibbs},    $F$ is   a Gibbs-Markov map  with respect to the $\leb_{\gamma_0} $ mod~0  partition $\mathcal P=\{\gamma_0\cap \Lambda_1,\gamma_0\cap \Lambda_2, \dots\}$       of~$\gamma_0\cap \Lambda$.   
Notice that
 \begin{equation}\label{eq.Rcompatible}
R\vert_{\gamma_0\cap\Lambda_i}=R\vert_{ \Lambda_i}=R_i,
 \end{equation}
Since $R$ is constant in the elements of $\mathcal P$, we can   consider the tower map $$T:\Delta\to\Delta$$ of the Gibbs-Markov map $F$ with recurrence time $R$. As before, we still denote   the reference measure on $\Delta$ by $m_{\gamma_0}$. 
 Since  $\gamma_0\cap\Lambda\subset\Lambda$, it   follows that, for all $\ell\ge0$,
 \begin{equation}\label{eq.hatell}
\Delta_\ell\subset \hat\Delta_\ell \qand T\vert_{\Delta_\ell}=\widehat T\vert_{\Delta_\ell},
\end{equation}
with $\hat\Delta_\ell$ and $\widehat T$ as in the beginning of this Section~\ref{hatower}.
Hence, it makes sense to consider   the map  
 \begin{equation}\label{eq.tribalistas}
 \begin{array}{rcccl}
 \Theta&\colon &  \hat\Delta & \longrightarrow &
 \Delta
 \\
& &(x,\ell)&\longmapsto & \left(\Theta_{\gamma_0}(x),\ell\right).
 \end{array}
 \end{equation}
It is straightforward to check that 
 \begin{equation}\label{eq.leozinho}
\widehat T\circ\Theta=\Theta\circ T,
\end{equation}
 thus  $\Theta$  being  a  semiconjugacy between the tower maps $\widehat T$ and $ T$. 
  Observe that $\Theta$ is not countable-to-one, and so we cannot invoke   \cite[Proposition 2.8]{B99b} to obtain  $h_{\hat{\nu}}(\widehat{T})=h_{{\nu}}({T})$. This will be deduced later, by mean of  natural extensions.
 The next result shows that     $(\widehat T,\hat\nu)$ is in fact an extension of $(T,\nu)$;   
see \cite[Proposition 4.13]{A20} for a proof.

\begin{proposition}\label{le.projmeasuredif}
If $\hat\nu$ is the ergodic SRB measure for $\widehat T$, then $ \Theta_*\hat\nu$ is the unique ergodic $T$-invariant probability measure absolutely continuous with respect to $ \leb_{\gamma_0}$.
\end{proposition}

Hence, the measure preserving systems $(f, \mu)$ and $(T,{\nu})$ are both factors of $(\widehat{T}, \hat{\nu})$, and we have  the  commuting diagram
\begin{equation*}
  \begin{CD}
    (M, \mu) @<\pi<< (\hat{\Delta},\hat{\nu}) @>\Theta>> (\Delta, \nu) \\
    @VV f V @V V \widehat{T} V @V   VTV \\
    (M, \mu) @<\pi<< (\hat{\Delta},\hat{\nu}) @>\Theta>> (\Delta, \nu)
  \end{CD}  
\end{equation*}

\section{Natural extensions}

Now, we turn our attention to the tower systems $(T, {\nu})$ and $(\widehat{T}, \hat{\nu})$. Heuristically, it is natural to expect these two systems have the same entropy, since we are in a certain sense just ignoring the stable direction, where no dynamical information is produced. The formal way in which we will deduce this fact is via   natural   extensions.  Let us briefly recall this concept, in a general setting. 
Let $\phi:X\to X$ be a measure preserving transformation of a probability measure space $(X,\mathcal A,\eta)$.
Set
 $$
 X^\#= \left\{(x_1,x_2,\dots)\in \prod_{i=1}^\infty X\colon \phi(x_{i+1})=x_i\right\}
 $$
 and   the map $\phi^\#: X^\#\to X^\#$, given  by 
\begin{equation*}
\phi^\#(x_1,x_2,\dots)=(\phi(x_1),x_1,x_2,\dots).
\end{equation*}
Consider the $\sigma$-algebra $\mathcal A^\#$ in $X^\#$ generated by  cylinders of the form 
$$
[A_1, \dots, A_k]=\left\{(x_1,x_2,\dots) \in X^{\#} \colon \,x_i \in A_i, \mbox{ for all } 1\le i\le k\right \}, 
$$
where $A_i \in \hat{\mathcal{A}}$, for all $1\le i\le k$. It is easily verified that $\phi^\#$ preserves the probability measure $\eta^\#$     defined in    the cylinders by
 $$\eta^{\#}([A_1,\dots,A_k])=\eta\left(A_k \cap \phi^{-1}(A_{k-1}) \cap \dots \cap \phi^{-k+1}(A_0)\right).$$
 Moreover, the map $\pi^\#:X^\#\to X$, given by $\pi^\#(x_1,x_2,\dots)=x_1$, is a semiconjugacy between $\phi^\# $ and $\phi$ and $\pi^\#_*\eta^\#=\eta$. 
The measure preserving system $(\phi^\#,\eta^\#)$ is called the \emph{natural extension} of $(\phi,\eta)$.
A classical result due to Rohlin gives that the entropies of these two measure preserving systems coincide, i.e.
\begin{equation}\label{le.nat1}
h_{\eta^\#}(\phi^\#)=h_{\eta}(\phi).
\end{equation}
 see  \cite[Section 3.3]{R64} or \cite[Section 9.9]{R67}.
For the natural extensions of the tower systems $(\widehat T,\hat\nu) $ and $(T,\nu)$, it is proved 
  in  \cite[Appendix B]{DWY12} that the
  transformation  $\Theta^\#:\hat\Delta^\#\to\Delta^\#$, given by
$$\Theta^\#((x_1,\ell_1),(x_2,\ell_2),\dots)=(\Theta(x_1,\ell_1),\Theta(x_2,\ell_2),\dots)
$$
 is an isomorphism of the measure preserving systems $(\widehat{T}^\#,\hat\nu^\#)$ and $(T^\#, \nu^\#)$.
  This in particular  implies that
  \begin{equation}\label{eq.toreex}
h_{\hat{\nu}^{\#}}(\widehat{T}^{\#})=h_{ {\nu}^{\#}}( {T}^{\#}).
\end{equation}

\section{Entropy of the original system}
In this section, we complete  the proof of Theorem~\ref{entropy formula}.
Since    $(\widehat{T},\hat \nu)$ is an extension of $(f,\mu)$ with countably many fibers, it follows from    \cite[Proposition 2.8]{B99b} that
\begin{equation} \label{eq.buzzi28}
h_\mu(f)= h_{\hat\nu}(\widehat{T}).
\end{equation}
Moreover, by~\eqref{le.nat1} and~\eqref{eq.toreex},
\begin{equation}\label{eq.entros}
h_{\hat{\nu}}(\widehat{T})=h_{\hat{\nu}^{\#}}(\widehat{T}^{\#})=h_{ {\nu}^{\#}}( {T}^{\#})=h_{{\nu}}({T}).
\end{equation}
Using  \eqref{eq.final2},~\eqref{eq.buzzi28} and \eqref{eq.entros} we get
\begin{equation}\label{eq.final22} 
h_\mu(f)<\infty \iff \int_{\Delta_0} R\log J_{F}dm_{\gamma_0}<\infty.
\end{equation}
On the other hand, using~\eqref{eq.jacompoe} and~\ref{Y5}, we obtain
\begin{equation}\label{eq.final23} 
\int_{\Delta_0} R\log J_{F}dm_{\gamma_0}<\infty
 \iff
 \int_{\Delta_0} R\log  |\det Df^{R}\vert_{\gamma_0}| dm_{\gamma_0}<\infty.
\end{equation}
The second item of Theorem~\ref{entropy formula} is then a consequence of~\eqref{eq.final22}  and~\eqref{eq.final23}.
Now, we prove  the first item of Theorem~\ref{entropy formula}. Assuming  $h_\mu(f)<\infty$, it follows from~\eqref{eq.final2}, \eqref{eq.buzzi28}, \eqref{eq.entros} and
   Proposition \ref{entrinduce} that  
$$
h_\mu(f)= \int_\Delta  \log J_T  d{\nu},
$$
and from Lemma~\ref{davids}, 
$$\displaystyle\int_\Delta  \log J_T \, d{\nu} =\frac1\rho   \int_{\gamma_0\cap\Lambda}   \log J_F \,  d \nu_0,$$
with $\rho$ as in~\eqref{eq.rho}. 
By the last two displayed formulas, the proof of Theorem \ref{entropy formula} will be complete  as soon as we get
\begin{equation}\label{eq.compara}
\displaystyle \int_{\gamma_0\cap\Lambda}   \log J_F \,  d \nu_0=\rho  \int_M \log |\det Df_u|\,d\mu.
\end{equation}
This will be  obtained in the next two lemmas. 
Recalling~\eqref{eq.rho},~\eqref{eq.Rcompatible} and  Proposition~\ref{le.SRBproj}, we have in this case
\begin{equation}\label{newro}
\rho 
=  \sum_{j\geq 0}(\Theta_{\gamma_0)_*}\mu_0(\{ R\vert_{\gamma_0\cap\Lambda}> j\})
= \sum_{j\geq 0} \mu_0(\{ R  > j\}).
\end{equation}

\begin{lemma} \label{le.last1}
$\displaystyle\int_{\gamma_0\cap\Lambda}   \log J_F \,d{\nu_0} = \int_{\Lambda}  \log  |\det Df^R_u|  \,d  \mu_0.$
\end{lemma}

\begin{proof}
Using~\eqref{eq.semifRF}, we easily see that, for every $x\in\Lambda$, 
$$
F (\Theta_{\gamma^u(x),\gamma_0}(x)) = \Theta_{\gamma^u(f^R(x)),\gamma_0}( f^R(x)) .
$$
It follows from \ref{Y5} that, for all $\gamma,\gamma'\in\Gamma^u$, the transformation $\Theta_{\gamma,\gamma'}$ has a Jacobian with respect to the measures $m_\gamma$  and $m_{\gamma'}$ on $\gamma$ and $\gamma'$, respectively. Using the Chain Rule  and applying logarithms, we have  for   each $x\in\Lambda$
\begin{equation}\label{eq.jurga0}
\log J_F(\Theta_{\gamma^u(x),\gamma_0}(x))+\log J_{\Theta_{\gamma^u(x),\gamma_0}}(x)
=
\log J_{\Theta_{\gamma^u(f^R(x)),\gamma_0}}(f^R(x))+ \log  |\det Df^R_u(x)|.
\end{equation}
Now, observing that  $\Theta_{\gamma^u(x),\gamma_0}(x)=\Theta_{\gamma_0}(x)$ and $(\Theta_{\gamma_0})_*\mu_0=\nu_0$ (recall Proposition~\ref{le.SRBproj}), we have
\begin{equation}\label{eq.jurga1}
\int_\Lambda \log J_F(\Theta_{\gamma^u(x),\gamma_0}(x)) d\mu_0(x)=\int_\Lambda \log J_F(\Theta_{\gamma_0}(x)) d\mu_0(x)
=\int_{\gamma_0\cap\Lambda} \log J_F(x) d\nu_0(x).
\end{equation}
Since the measure $\mu_0$ is $f^R$-invariant, we also have
\begin{equation}\label{eq.jurga3}
\int_\Lambda\log J_{\Theta_{\gamma^u(x),\gamma_0}}(x)d\mu_0(x)
=
\int_\Lambda \log J_{\Theta_{\gamma^u(f^R(x)),\gamma_0}}(f^R(x))d\mu_0(x).
\end{equation}
From~\eqref{eq.jurga0}, \eqref{eq.jurga1} and \eqref{eq.jurga3} we get the conclusion.
\end{proof}

\begin{lemma}\label{le.igual}
$\displaystyle\int_\Lambda\log  |\det Df^R_u|\,d \mu_0=\rho   \int_M\log |\det Df_u|\,d\mu.$

\end{lemma}

\begin{proof} Set for each $n\ge 1$
 $$P_n=\{x\in\Lambda\colon R(x)=n\}.$$
 By the chain rule, we have for each $x\in P_n$,
 $$ \det Df^R_u(x)=\det Df_u(f^{n-1}(x))\cdots \det Df_u(f(x))\cdot \det Df_u(x).$$
It follows that
 \begin{eqnarray*}
 \int_\Lambda\log  |\det Df^R_u| d\mu_0 &= &\sum_{n=1}^\infty\int_{P_n}\log |\det Df^R_u| \,d\mu_0\\
 &= &\sum_{n=1}^\infty\int_{M}\log   |\det Df^R_u| \,d(\mu_0\vert P_n)\\
  &=&
 \sum_{n=1}^\infty\sum_{j=0}^{n-1}\int_M\log  |\det Df_u|\circ f^j\,d(\mu_0\vert P_n)\\
  &=&
 \sum_{n=1}^\infty\sum_{j=0}^{n-1}\int_M\log  |\det Df_u|\,d\left(f^{j}_*(\mu_0\vert P_{n})\right)\\
  &=&
 \sum_{n=0}^\infty\int_M\log   |\det Df_u|\,d\left(f^{n}_*(\mu_0\vert \{R>n\})\right)\\
  &=&
  \int_M\log   |\det Df_u|\,d\left(\sum_{n=0}^\infty f^{n}_*(\mu_0\vert
\{R>n\})\right).
 \end{eqnarray*}
By Theorem~\ref{th.SRB},
$$
\int_M\log   |\det Df_u|\,d\left(\sum_{n=0}^\infty f^{n}_*(\mu_0\vert \{R>n\})\right)= {\sum_{j= 0}^\infty\mu_0(\{  R > j\})}   \int_M\log  |\det Df_u| d\mu.
$$
Using~\eqref{newro}, we finish  the proof.
 \end{proof}

\section{Infinite entropy}\label{se.systemno2}

Here, we adapt the example in Section~\ref{se.systemno} to build a piecewise $C^\infty$ diffeomorphism  with an SRB measure given by  a Young set  for which the entropy formula does not hold. Consider the following objects, as described in Section~\ref{se.systemno}:
\begin{enumerate}
\item 
the interval  $\Delta_0$ with finite Lebesgue measure $m$;
\item 
 the $m$ mod 0 partition $\mathcal P=\{\omega_n\}_{n\ge 2}$ of $\Delta_0$;
\item 
the  piecewise $C^\infty$  Gibbs-Markov map $F:\Delta_0\to\Delta_0$ with   associated partition $\mathcal P$;
\item 
the    function $R:\Delta_0 \to \mathbb N$ such that $R\vert_{\omega_n}=n$, for all $n\ge 2$. 
\end{enumerate}
We are going to introduce an extra contracting direction and slightly modify the  standard tower construction  associated with these objects   in order to obtain a piecewise $C^\infty$  diffeomorphism  $f: M\to M$ of   a  rectangle $M$ in $\mathbb R^2$. In addition, we will show that  $f$ has an ergodic SRB measure $\mu$ given by a Young set   such that
\begin{equation}\label{notentro2}
h_\mu(f)=\infty\qand \int_M  \log |\det Df_u | \, d{\mu}<\infty.
\end{equation}
Set
 $$ \Delta=\big\{ (x,\ell)\colon \text{$x
\in \Delta_0$ and $0\leq \ell < R(x)$} \big\}$$ and $$M=\Delta \times [0,1].$$ 
As observed   in   Section~\ref{se.systemno},  the set $\Delta$ can   be identified with an interval in $\mathbb R$, and so $M$ can be  identified with a two-dimensional rectangle.
Taking some $ \lambda\in(0, 1/2]$,  consider  $f:  M\to M$   given  by
 $$
f(x,\ell, y)= 
\begin{cases}
    \left(x,\ell+1, \displaystyle \lambda y\right), & \hbox{if $
\ell<R(x)-1$;} \\
    (F(x),0,\lambda +\lambda^2+\cdots+\lambda^{R(x)-1} +\lambda^{R(x)}y ), & \hbox{if $\ell=R(x)-1$.} \\
\end{cases}%
 $$
For each $n\ge 2$ and $\ell\ge 0$, consider  the open set
 $$M_{n,\ell} =\omega_n\times\{\ell\}\times (0,1).$$
It is easily verified that  these sets are pairwise disjoint and $\cup_{n,\ell} \closure{M}_{n,\ell}=M$. Let us now show that   $f\vert_{\cup_{n,\ell} M_{n,\ell}}$ is a  $C^{\infty}$ diffeomorphism onto its image. 
Recalling that $F$ is affine   on each $\omega_n=\{x\in\Delta_0\colon R(x)=n\}$, it is enough to prove that $f$ is injective. For this, we just need to show that, given  $x_n\in \omega_n$, $x_k\in\omega_k$ with $n\neq k$ and $y,z\in(0,1)$, we have
 $$
 f(x_n,n-1,y)\neq f(x_k,k-1,z).
 $$
Assume for definiteness $n>k$. Since
$$
 f(x_n,n-1,y)=(F(x_n),0,\lambda +\lambda^2+\cdots+\lambda^{n-1} +\lambda^{n}y )
 $$
and
$$
 f(x_n,k-1,z)=(F(x_k),0,\lambda +\lambda^2+\cdots+\lambda^{k-1} +\lambda^{k}z )
 $$
we have that the difference of the third coordinates of these two images is  
$$
\lambda^k+\cdots +\lambda^{n-1}+\lambda^n y-\lambda^k z > \lambda^ny>0.
$$
This implies that  $f$ is injective. 
Now,  
taking
$$\Lambda=\Delta_0\times\{0\}\times [0,1]\subset M
 $$
 we have that $\Lambda$ has a product structure given by the continuous  families of disks
 $$
 \Gamma^s=\left\{\{x\}\times \{0\}\times [0,1]: x\in \Delta_0\right\}
 \qand \Gamma^u= \{\Delta_0\times \{0\}\times \{y\}\colon y\in [0,1]\}.
 $$
 Set for all $i\ge 2$
 $$\Lambda_i=\omega_i\times\{0\}\times [0,1]\qand R_i=i.  
 $$
It is easily verified  that $(\Lambda_i)_{i\ge2}$ is a family of pairwise disjoint $s$-subsets of $\Lambda$ and  the return map $f^R$ associated with these objects  as in
 \eqref{defRfR} satisfies  conditions   \ref{Y1}-\ref{Y5} in Section~\ref{section.GMY}, thus showing that $\Lambda$ is a Young set with integrable recurrence times. Let~$\mu$ be the SRB measure for $f$ given by Theorem~\ref{th.SRB}.   Note that the quotient map associated with $f^R$ on the disk $\gamma_0=\Delta_0\times\{0\}\times\{0\}$  is   the map $F$, under the natural identification of~$\gamma_0$ and $\Delta_0$. Moreover,  $\rho_{\gamma_1,\gamma_0}=1$, for all $\gamma_1\in\Gamma^u$. It follows from~\eqref{tres}, \eqref{eq.jacompoe} and Lemma~\ref{le.sums} 
that   $$\int_{\gamma_0\cap\Lambda} R\,\log |\det D{f^R}\vert_{\gamma_0}| \,dm_{\gamma_0}=\int_{\Delta_0} R\,\log J_F \,dm_{\gamma_0}=\infty .$$
Using  Theorem~\ref{entropy formula}, we obtain
 $h_\mu(f)=\infty$. On the other hand,   it follows from \eqref{dois}, \eqref{eq.compara}  and Lemma~\ref{le.sums}  that
 $$
\int_M  \log |\det Df_u | \, d{\mu}=\frac1\rho \int_{\Delta_0}\log J_F dm<\infty.
 $$
This concludes the proof of~\eqref{notentro2}.

\section{Applications}\label{se.app2}

In the next subsections we apply Theorem~\ref{entropy formula} and Corollary~\ref{co.entform1} to some classes of piecewise smooth diffeomorphisms with SRB measures given by Young sets. The case of billiard maps has already been considered  in \cite{LS82} through the approach in \cite{KSL86}.
\subsection{Piecewise hyperbolic maps}
Here, we apply our results to  a class of piecewise hyperbolic diffeomorphisms    studied   by Young  in \cite{Y98}, in  dimension two,  and   by Chernov  
in \cite{C99b}, in any finite dimension. Concrete examples of such systems  include  the family of Lozi-like mappings   in \cite{Y85a}; see also \cite{CL84,M80,R83b}.
Let $M$ be a compact $d$-dimensional Riemannian manifold, for some  $d\ge 2$, possibly with boundary. Assume that  $f:M\rightarrow M$  satisfies the following conditions:

\begin{enumerate}
\item  $f$ is a \emph{piecewise $C^2$ diffeomorphism} from $M$ into itself: there is a finite number of pairwise disjoint open regions $ (M_i)_i$, with $M=\cup_i \closure{M}_i$, whose boundaries are $d-1$ submanifolds  such that  
\begin{enumerate} 
\item $f\vert_{(\cup_i M_i)}$ is injective;
\item $f\vert_{M_i}$ can be extended to a $C^2$ diffeomorphism of $\closure{M}_i$ onto its image, for all $i$.
\end{enumerate}
\item $f$ is \emph{uniformly hyperbolic}: there exist $Df$-invariant cone families $ \mathcal C^u$ and $ \mathcal C^s$ on $M$ and $\lambda>1$ such that, for all $i$ and $x\in \closure M_i$
\begin{enumerate}
\item $|Df_x(v)|\geq \lambda |v|$, for all $v\in \mathcal C_x^u$;
\item $|Df^{-1}_x(v)|\geq \lambda |v|$, for all $v\in \mathcal C_x^s$. 
\end{enumerate}
\end{enumerate}
We shall refer to $\mathcal{S}=M\setminus \cup_i M_i$ as the \textit{singularity set}. Note that  
we allow $
\cup_i f(\closure{M}_i) \subsetneq M
$,
 so that $M$ can be a trapping region for an attractor. For $n\ge 1$, denote by
 $$\mathcal S_n=\mathcal S \cup f^{-1}(\mathcal S)\cup\cdots\cup f^{-n+1}(\mathcal S)$$
 the singularity set for $f^n$. 


\begin{enumerate}
\setcounter{enumi}{2}

\item The angle between  $\mathcal{S}$ and $\mathcal C^u$ is bounded away from 0.
\item There is $n\ge 1$ such that the multiplicity of any point in $\mathcal S^{n}$ is smaller  that $\lambda^n-1$. 
\end{enumerate}
The results in~\cite{C99b,Y98} show  that $f$ has an ergodic SRB measure $\mu$ given by a Young set with the tail  of recurrence times decaying exponentially fast. In particular, the square of the recurrence time function is integrable with respect to Lebesgue measure on any unstable leaf in the family that defines the  Young set.   Since $| \det Df_u|$ is bounded, it follows from
Corollary~\ref{co.entform1} that 
$$
h_\mu(f)=\int_M   \log |\det Df_u| \,  d\mu<\infty.
$$


\subsection{Billiard maps}

\addtocontents{toc}{\protect\setcounter{tocdepth}{3}}

Here, we apply the main  results of this part to some classes of  billiard maps. Maps of this type have  been        introduced by Sinai in~\cite{S70a}, and  can be described as follows. Let  $\Gamma_1,\dots,\Gamma_d$ be  pairwise disjoint simply connected $C^3$ curves in the torus $\mathbb T^2$, which can be interpreted as the boundaries of scatters. 
Consider the billiard flow 
on the domain $\mathbb T^2\setminus \cup_{i}\inte(\Gamma_i)$, where each $\inte(\Gamma_i)$ stands for the interior of the curve~$\Gamma_i$,
generated by  the motion of  point particles  
 traveling at  unit speed  and having  elastic reflections at   the boundary~$\cup_i \Gamma_i$. This flow has the  Poincaré section $$M=\cup_i \Gamma_i\times [-\pi/2,\pi/2],$$  the first coordinate giving  the collision point in $\cup_i \Gamma_i$ and the second one   the angle of the trajectory with the normal to $\cup_i \Gamma_i$ at the collision point.
We are interested in the first return map   $f:M\to M$, known as the \emph{billiard map}. This is essentially a piecewise hyperbolic diffeomorphism as in Subsection~\ref{se.app2}, with the   difference     that $ Df $ is not bounded, due to the   tangencial reflections corresponding to angles $ \pm\pi/2$. Considering  $  \cup_i \Gamma_i$ parametrised by arc length    $x$   and  $\theta \in  [-\pi/2,\pi/2]$,
it is   known that $f$ preserves an ergodic  measure $\mu$ given by 
\begin{equation}\label{eq.srb}
d\mu=k\cos\theta dxd\theta, 
\end{equation}
 where $k>0$ is a normalizing constant; see e.g. 
\cite[Section 2.12]{CM06}. In the next two subsections, we use  Theorem~\ref{entropy formula} to deduce the entropy formula for $\mu$ in two special cases of billiards. Due to the unboundedness of $Df$, we cannot apply Corollary~\ref{co.entform1} in this context.

 \subsubsection{Dispersing billiards}
 

Assume that the curves $\Gamma_1,\dots,\Gamma_d$  have   strictly positive curvature. In this case, 
Young  proved in  \cite{Y98} that the measure $\mu$ is given by a Young set  with exponential tail of recurrence times if the  billiard has  \emph{finite horizon}, i.e. when the time between collisions is uniformly bounded.  This conclusion was  extended by Chernov to billiards with  infinite horizon in~\cite{C99}.
In order to apply Theorem~\ref{entropy formula}, we need to ensure that the entropy $h_\mu(f)$ is finite. This follows  easily from \cite[Lemma 3.6]{BD20} in the finite horizon case. In general, this may be deduced  from  \cite[Corollary 2.4 \& Example 3.1]{C91}.
Therefore, using Theorem~\ref{entropy formula}, we get
 $$
h_\mu(f)=\int_M   \log |\det Df_u| d\mu<\infty.
$$

\subsubsection{Semi-dispersing billiards}
Chernov and Zhang 
considered   a class of billiards  for which the the curvature of the curves $\Gamma_1,\dots,\Gamma_d$ vanishes at some points. As observed in~\cite{CZ05}, if there is no periodic trajectory that hits the boundary at flat points only, then a certain power of the collision map is uniformly hyperbolic. 
To avoid this situation, and for simplicity, they  assume that there is one such periodic trajectory of period two that runs between two flat points.
Moreover, the boundary near these flat points is given by 
 $$y=\pm (1+|x|^\alpha),\quad \alpha>2,$$
 in some rectangular coordinate system   $(x,y)\in \mathbb R^2$. As a byproduct of the results by Chernov and Zhang, we have   that  the measure $\mu$  is given by a Young set  with  tail of recurrence times decaying as
$\mathcal O((\log n)^{\beta+1}/n^{\beta+1})$, for  $\beta=(\alpha +2)/(\alpha-2)$; see conditions (F1)-(F2) in~\cite[Section 3]{CZ05} and the proof of \cite[Theorem 4]{CZ05a}. This implies that  the recurrence times are integrable for all $\alpha>2$. In order to apply Theorem~\ref{entropy formula}, we need to ensure that $h_\mu(f)<\infty$.  This follows from  \cite[Corollary 2.4]{C91} and  \cite[Corollary 1.3]{S89}. Therefore, using Theorem~\ref{entropy formula}, we get
 $$
h_\mu(f)=\int_M   \log |\det Df_u| d\mu<\infty.
$$

\addtocontents{toc}{\protect\setcounter{tocdepth}{0}}

\bibliographystyle{acm}


\end{document}